\documentclass[11pt]{amsart}
\usepackage{graphicx}
\usepackage{tgpagella}
\usepackage{euler}
\usepackage[T1]{fontenc}
\usepackage{amssymb}
\usepackage{amsmath,amsthm}
\usepackage[hidelinks]{hyperref}
\usepackage{amsthm}
\usepackage[english]{babel}
\usepackage{mathrsfs}
\usepackage{eucal}

\newtheorem{main}{Theorem}

\newtheorem{theorem}{Theorem}[section]
\newtheorem{lem}[theorem]{Lemma}
\newtheorem{prop}[theorem]{Proposition}
\newtheorem{Conjecture}[theorem]{Conjecture}
\newtheorem{cor}[theorem]{Corollary}

\theoremstyle{definition}
\newtheorem{definition}[theorem]{Definition}
\newtheorem{notation}[theorem]{Notation}
\newtheorem{Remark}[theorem]{Remark}

\newtheorem{examples}[theorem]{Examples}
\def\ca{\curvearrowright}
\def\ra{\rightarrow}
\def\e{\epsilon}
\def\la{\lambda}
\def\La{\Lambda}

\def\ve{\varepsilon}

\def\de{\delta}

\def\g{\gamma}

\def\G{\Gamma}
\def\Cal{\mathcal}

\def\wh{\widehat}
\def\hla{\hat\la}

\numberwithin{equation}{section}

\def\mod{\rm Mod}
\def\mod{\mathrm{Mod}}
\def\pmod{\mathrm{PMod}}
\def\cali{\mathcal{I}}
\def\calk{\mathcal{K}}

\begin{document}

\title[Von Neumann algebras arising from surface braid groups]{Primeness results for von Neumann algebras associated with surface braid groups}
\author[I. Chifan]{Ionut Chifan}
\address{Department of Mathematics, The University of Iowa, 14 MacLean Hall, IA  
52242, USA and IMAR, Bucharest, Romania}
\email{ionut-chifan@uiowa.edu}
\thanks{First author was supported in part by the Old Gold Fellowship, University of Iowa and by the NSF Grants \#1263982 and \#1301370.}
\author[Y. Kida]{Yoshikata Kida}
\address{Graduate School of Mathematical Sciences, The University of Tokyo, Komaba, Tokyo 153-8914, Japan}
\email{kida@ms.u-tokyo.ac.jp}
\thanks{Second author was supported by JSPS Grant-in-Aid for Young Scientists (B), No.25800063.}
\author[S. Pant]{Sujan Pant}
\address{Department of Mathematics, The University of Iowa, 14 MacLean Hall, IA
52242, USA}
\email{sujan-pant@uiowa.edu}
\date{\today }
\dedicatory{}

\begin{abstract} In this paper we introduce a new class of non-amenable groups denoted by ${\bf NC}_1 \cap {\bf Quot}(\mathcal C_{rss})$ which give rise to \emph{prime} von Neumann algebras. This means that for every $\G\in {\bf NC}_1 \cap {\bf Quot}(\mathcal C_{rss})$ its group  von Neumann algebra $L(\G)$ cannot be decomposed as a tensor product of diffuse von Neumann algebras. We show ${\bf NC}_1 \cap {\bf Quot}(\mathcal C_{rss})$ is fairly large as it contains many examples of groups intensively studied in various areas of mathematics, notably: all infinite central quotients of pure surface braid groups; all mapping class groups of (punctured) surfaces of genus $0,1,2$; most Torelli groups and Johnson kernels of (punctured) surfaces of genus $0,1,2$; and, all groups hyperbolic relative to finite families of residually finite, exact, infinite, proper subgroups. 

\end{abstract}

\maketitle


\section{Introduction}

In pioneering work \cite{Po81} Sorin Popa discovered that the (non-separable) factors $L(\mathbb F)$ arising from uncountably generated free groups $\mathbb F$ are \emph{prime}, i.e., $L(\mathbb F)$ cannot be decomposed as a tensor product of diffuse factors. Much later, using Voiculescu's influential free probability theory \cite{VDN92,V94,V96}, Liming Ge  was able to show primeness for all factors associated with countably generated, non-abelian free groups as well, \cite{G98}. In the context of free probability other examples of prime factors were subsequently unveiled \cite{Shl00,Shl04,Ju07}.

By developing a different perspective, largely based on $C^*$-algebraic methods, Narutaka Ozawa obtained a far-reaching generalization of these results by showing that all  factors $L(\G)$ associated with non-elementary hyperbolic groups $\G$ are in fact \emph{solid}, i.e., for every diffuse, amenable subalgebra $A\subset L(\G)$ its relative commutant $A'\cap L(\G)$ is again amenable; in particular, it follows that $L(\G)$ is prime. Notice that Ozawa's solidity result holds for all factors associated with bi-exact groups \cite{Oz03,Oz05,BO08}.  

In the early 2000's  Popa introduced a completely new conceptual framework to study von Neumann algebras, now termed \emph{Popa's deformation/rigidity theory}. This novel approach has generated spectacular progress over the last decade , leading to complete solutions to many longstanding open problems in the classification of von Neumann algebras and equivalence relations arising from group actions, \cite{Po01,Po03,Po04,IPP05,Po06,Po06b}. The theory develops a powerful technical apparatus designed to incorporate meaningful cohomological, geometric, and algebraic information of a group and its actions in the analytic  context of von Neumann algebras. Overtime, these methods became more and more precise and sophisticated and reveled unprecedented connections with cohomological, geometric, and dynamical aspects in group theory. 

These techniques are very suitable to study the primeness phenomenon as well. Indeed, using his free malleable deformations in combination with a novel spectral gap argument, Popa was able to find a new, elementary proof for Ozawa's solidity result for the non-amenable free group factors, \cite{Po06}. His approach laid out the foundations for many important subsequent developments regarding the algebraic structure of factors. For instance, it provided the correct insight which later allowed Ozawa and Popa to show in a remarkable work \cite{OP07} that all non-amenable free group factors $L(\mathbb F)$ are in fact \emph{strongly solid}, i.e., for every diffuse amenable subalgebra $A\subset L(\mathbb F)$, its normalizing group $\mathcal N_{L(\mathbb F)}(A)=\{ u\in\mathcal U(L(\mathbb F)) \,:\, uAu^*=A\}$ generates an amenable von Neumann subalgebra in $L(\mathbb F)$---a result of great influence for the entire subsequent development on the classification of normalizers of algebras in many classes of factors.

Exploiting a new viewpoint which originates in his recent ingenious study of unbounded derivations on von Neumann algebras, Jesse Peterson further showed that every non-amenable, icc  group with positive first $\ell^2$-Betti number gives rise to a prime factor, \cite{Pe06}. 

These results along with Ozawa's earlier solidity results have spawned a rich activity in the classification of von Neumann algebras. Numerous technical outgrowth of these methods by several authors have led overtime to the discovery to many striking structural results including primeness, (strong) solidity, uniqueness of Cartan subalgebra, and beyond for large classes of von Neumann algebras, \cite{Po08,OP07,OP08,CH08,CI08,Pe09,PV09,FV10,Io10,IPV10,HPV10,CP10,Si10,Va10,Fi11,CS11,CSU11,Io11,PV11,HV12,PV12,Io12,Bo12,BHR12,Is12,BV13,Va13,Is14, CIK13,VV14,BC14}. 

\subsection{Statements of main results} In this paper we introduce new families of groups which give rise to prime von Neumann algebras. Many of these groups are intensively studied in various areas of Mathematics, especially topology and geometric group theory. Over time, via deep topological and geometric methods, many strong classification results emerged regarding the structure of these groups in both discrete  and measurable setting. However, momentarily, little is known about the structure of the von Neumann algebras associated with these groups and this paper initiates a study in this direction. Formally,  we define our class of groups as follows:

\begin{definition} A group $\G$ belongs to class ${\bf NC}_1 \cap {\bf Quot}(\mathcal C_{rss})$   if the following two conditions are satisfied:
\vskip 0.03in
 
 $\bullet$ {\bf NC$_1:$} $\, \G$ is non-amenable and admits an unbounded quasi-cocycle valued into one of its mixing, weakly-$\ell^2$, orthogonal representations (see Section \ref{sec: nc1} for relevant definitions);

$\bullet$ {\bf Quot$(\mathcal C_{rss}):$} $\,\G$ is a finite-step extension of groups belonging to $\mathcal C_{rss}$---the collection of all non-elementary hyperbolic groups and non-amenable, non-trivial free products of exact groups (see Sections \ref{sec: dichotomy} and \ref{sec: quot} for relevant definitions).
\end{definition}

While at a first look this definition may seem a little restrictive and not entirely illuminating, we will show however that $\Cal P$ is fairly large, containing many important examples of groups, such as:
\begin{enumerate}
\item [a)] Any infinite, central quotient of the pure braid group $PB_n(S_{g,k})$ of $n$ strands on a connected, compact and orientable surface $S_{g,k}$ of genus $g$ with $k$ boundary components---in particular, all surface pure braid groups $PB_n(S_{g,k})$, for $n\geq 1$ and either $g=1$ and $k\geq 1$ or $g\geq 2$ and $k\geq 0$;
\item [b)] Any mapping class group $\mod(S_{g,k})$, for $0\leq g\leq 2$ and $2g+k\geq 4$;
 \item [c)] Any Torelli group $\Cal I (S_{g, k})$ and Johnson kernel $\Cal K (S_{g, k})$, for $g=1, 2$ and $2g+k\geq 4$;
 \item [d)] Any group that is hyperbolic relative to a finite family of exact, residually finite, infinite, proper subgroups. 
\end{enumerate} 
Also we notice that from Theorem \ref{NCpreserve}, Proposition \ref{quot}, and \cite[Proposition 4.7]{VV14} it follows that {\bf ${\bf NC}_1 \cap {\bf Quot}(\mathcal C_{rss})$} is closed under commensuration. For the proofs of these results as well as other basic properties of this class we refer the reader to Sections \ref{sec: quot} and \ref{sec: nc1} in the sequel.

The central result of the paper shows that all groups in ${\bf NC}_1 \cap {\bf Quot}(\mathcal C_{rss})$  give rise to prime von Neumann algebras; in particular, for any $\G\in  {\bf NC}_1 \cap {\bf Quot}(\mathcal C_{rss})$ its von Neumann algebra \emph{cannot} be decomposed as $L(\G)= L(\Omega\times \Sigma)$, for any infinite groups $\Omega$ and $\Sigma$. 
\begin{main}\label{main1}
Let $\G$ be a group that can be realized as a finite-by-(${\bf NC}_1 \cap {\bf Quot}(\mathcal C_{rss})$) group.  Denote by $L(\G)$ its corresponding von Neumann algebra. If $p \in L(\G)$ is  a nonzero projection, then any two diffuse, commuting subalgebras $B,C\subseteq pL(\G)p$  generate together a von Neumann subalgebra $B\vee C$ which has infinite Pimsner-Popa index in $pL(\G)p$. In particular, $L(\G)$ is prime and hence $L(\G)\ncong L(\Omega\times \Sigma)$, for any infinite groups $\Omega$ and $\Sigma$. 
\end{main}
In this form the result is sharp, as in general there are groups $\G$ in ${\bf NC}_1 \cap {\bf Quot}(\mathcal C_{rss})$ whose algebras $L(\G)$ \emph{do} contain commuting, non-amenable, diffuse subalgebras which, together generate subalgebras in $L(\G)$ of  infinite index. To see some basic examples, consider $B_n$ to be the braid group on $n\geq 6$ strands and denote by $Z$ its center. By \cite[Section 9.2]{FM11}, the quotient $\G=B_n/Z$ can be realized as a subgroup of index $n$ inside the mapping class group of a $(n+1)$-punctured surface of genus zero and hence Theorem \ref{NCpreserve} and Examples \ref{examplesNC} c)  further imply that $\G \in {\bf NC}_1$. Moreover, using Birman short exact sequence, Corollary \ref{surjections} shows  that $\G$ is a $(n-2)$-step extension of non-abelian free groups and hence $\G \in \mathcal P$. On the other hand, notice that  $B_n$ contains mixed braid subgroups of the form $B_p\times B_q< B_n$, where $p+q=n$ with $p,q\geq 3$. Then one can check that the quotients $\G_1=B_p/Z$ and $\G_2=B_q/Z$ are commuting, non-amenable subgroups of $\G$ which together generate a subgroup  $\langle \G_1, \G_2\rangle <\G$ of infinite index. This canonically implies that $L(\G_1)$ and $L(\G_2)$ are commuting, non-amenable, diffuse subalgebras of $L(\G)$ which together generate a von Neumann subalgebra $ L(\langle \G_1, \G_2\rangle)\subset L(\G)$ of infinite index. Notice that, all such groups in ${\bf NC}_1 \cap {\bf Quot}(\mathcal C_{rss})$ will give rise to prime von Neumann algebras which are not solid in the sense of Ozawa, \cite{Oz03}.  

We believe that  Theorem \ref{main1} can be further improved by showing that the algebra $B\vee C$ is actually never co-amenable inside $pL(\G)p$, \cite{Po86}.  Notice that this will follow verbatim from our current proofs if one will be able to show an analogue of Proposition \ref{masa} in the context of co-amenable inclusions rather that finite index inclusions.   

The proof of our result is based on Popa's deformation/rigidity theory and is obtained by induction on $n$ where $\G\in {\bf NC}_1\cap {\bf Quot}_n(\mathcal C_{rss})$. For the induction step we use in an essential way recent, powerful results due to Popa and Vaes \cite{PV11,PV12} and to Ioana \cite{Io12} regarding the classification of normalizers of subalgebras in von Neumann algebras arising from actions by non-elementary hyperbolic groups  and by free product groups, respectively. Assuming by contradiction that $B\vee C\subseteq pL(\G)p$ has finite index then condition $\G\in {\bf Quot}_n(\mathcal C_{rss})$ enables us to employ these structural results, via the methods developed in \cite{CIK13}, to intertwine $B$ (or $C$) onto a subalgebra $B_0\subseteq  qL(\Sigma)q \subset qL(\G)q$, where $\Sigma\lhd\G$ is a normal subgroup satisfying $\Sigma \in {\bf Quot}_{n-1}(\mathcal C_{rss})$ and $q\in B_0$ is a nonzero sub-projection of $p$. Moreover there exists a subalgebra $C_0\subset qL(\Sigma)q$ commuting to $B_0$ such that $B_0\vee C_0\subseteq qL(\Sigma)q$ has finite index. Since $\G\in {\bf NC}_1$ then by \cite[Theorem 2.1]{CSU13} we have $\Sigma \in {\bf NC}_1\cap {\bf Quot}_{n-1}(\mathcal C_{rss})$ and by the induction hypothesis one can find a nonzero corner $r B_0 r$ of finite index in $r L(\Sigma)r$. On the other hand, since $\G\in {\bf NC}_1$, then a spectral gap argument shows that the corresponding weak deformations $V_t$ on $L(\G)$ arising from an unbounded quasi-cocycle on $\G$ \cite{CS11} will converge uniformly to the identity on the unit ball $(B)_1$. From this, developing new aspects in the infinitesimal analysis of $V_t$ (Section \ref{sec: deformations}) we further show that $V_t$ converges uniformly to the identity on the unit ball $(rB_0r)_1$ and by the finite index assumption it follows that a $V_t$ has a uniform decay on the unit ball $(r(L\Sigma) r)_1$. Hence the quasi-cocycle is bounded on $\Sigma$ and by \cite[Theorem 2.1]{CSU13} it is bounded on $\G$, which is a contradiction; thus $B\vee C$ must have infinite index in $pL(\G)p$.
 

Notice that, a spectral gap argument \cite[Proposition 1.7 (3)]{CS11} shows that for any group $\G\in  {\bf NC}_1$, any two commuting, infinite subgroups $\G_1, \G_2<\G$ generate an infinite index subgroup $\langle \G_1,\G_2 \rangle$  of $\G$ (in other words, $\G$ is not presentable by products). This should be seen as evidence supporting the far-reaching conjecture that Theorem \ref{main1} actually holds for all groups $\G$ satisfying only condition {\bf NC}$_1$ (even, without the mixing assumption on the representation). However, from a technical point of view, a successful implementation of this argument in the von Neumann algebra setting seems out of reach momentarily and depending heavily on investigating new aspects of the infinitesimal analysis of the weak deformations arising from quasi-cocycles. Unlike in the case of $1$-cocycles, the weak deformations arising from quasi-cocycles of groups seem to lack good averaging and uniform bimodularity properties which makes their analysis quite difficult. In our situation some of these difficulties could be by-passed through the knowledge that $\G$ admits a ``finite resolution'' by groups in $\mathcal C_{rss}$. Hence our result can be viewed as a first instance when knowing a little bit more information about the group (in addition of being in ${\bf NC}_1$) could decisively enhance the analysis on the weak deformations to conclude primeness results for many such group von Neumann algebras. It is then conceivable that there is actually an entire spectrum of such properties and a thorough investigation may reveal interesting results in this direction. 
 
As a byproduct of these methods, we obtain new applications of deformation/rigidity techniques to the algebraic structure of groups. Indeed, outgrowths of our methods in combination with the techniques developed in \cite{CSU13} allows us to deduce a result which  complements \cite[Theorem 3.5]{CSU13} in the case of groups satisfying condition ${\bf NC}_1$ above. 

\begin{main} For every $ \G\in {\bf NC}_1$ there exists a short exact sequence of groups
 $1\ra F \ra \G\ra \G_0\ra 1$, where $F$ is a finite and $\G_o$ is infinite conjugacy class. In particular, if $\G$ is assumed torsion free then $\G$ is infinite conjugacy class and non-inner amenable.
\end{main}

In particular, this provides a more quasi-cohomological explanation for some recent results on non-inner amenability of acylindrically hyperbolic groups by Dahmani, Guirardel, and Osin \cite{DGO11}, Osin \cite{Os13}, and Minasyan and Osin \cite{MO13}. This result also implies that every groups in ${\bf NC}_1 \cap {\bf Quot}(\mathcal C_{rss})$ gives von Neumann algebra with finite dimensional center, result that is implicitly used in the proof of Theorem \ref{main1}.




\subsection{Notations}

In this section we establish some notions that we will be used throughout the paper. 

A {\it tracial von Neumann algebra} $(M,\tau)$ is a pair that consists of a von Neumann algebra $M$ and a faithful normal tracial state $\tau$. We denote by $M_{+}$ the set of all positive elements $x\in M$ and by $\mathcal Z(M)$ the center of $M$.
For $x\in M$, we denote by $\|x\|$ the operator norm of $x$, and by $\|x\|_2=\sqrt{\tau(x^*x)}$ the 2-$norm$ of $x$.
Throughout the paper, we denote by $L^2(M)$ the Hilbert space obtained by completing $M$ with respect to $\|\cdot\|_2$, and consider the standard representation $M\subset\mathbb B(L^2(M))$.

If $M$ is a von Neumann algebra together with a subset $S\subset M$, then a state $\phi: M\rightarrow\mathbb C$ is called $S$-{\it central} if $\phi(xT)=\phi(Tx)$, for all $T\in M$ and $x\in S$. A tracial von Neumann algebra $(M,\tau)$ is called {\it amenable} if there exists an  $M$-central state $\phi:\mathbb B(L^2(M))\rightarrow\mathbb C$ such that $\phi(x)=\tau(x)$, for all $x\in M$. By a well-known theorem of A. Connes \cite{Co75}, $(M,\tau)$ is amenable if and only if it is approximately finite dimensional.

All inclusions of von Neumann algebras that appear in the paper are assumed to be unital unless specified otherwise. Let $(M,\tau)$ be a tracial von Neumann algebra and $P\subset M$ a von Neumann subalgebra. {\it Jones's basic construction} $\langle M,e_P\rangle\subset \mathbb B(L^2(M))$ is the von Neumann algebra generated by $M$ and the orthogonal projection $e_P:L^2(M)\rightarrow L^2(P)$. It is endowed with a faithful semifinite trace $Tr$ given by $Tr(xe_Py)=\tau(xy)$, for all $x,y\in M$. Also, we note that $E_P:={e_P}_{|M}:M\rightarrow P$ is the unique $\tau$-preserving conditional expectation onto $P$.

We say that $P\subset M$ is a {\it masa} if it is a maximal abelian $*$-subalgebra. 
The {\it normalizer of $P$ inside $M$}, denoted by $\mathcal N_{M}(P)$, is the set of all unitaries $u\in M$ such that $uPu^*=P$. We say that $P$ is {\it regular} in $M$ if $\mathcal N_{M}(P)''=M$.

Often if $M$ is a von Neumann algebra we denote by $M^{h}$ its hermitian part. If $S\subseteq M$ is a subset and $c>0$ we denote  by $(S)_c$ the set of all elements in $S$ whose operatorial norm does not exceed $c$.

Also if $B,C \subseteq M$ are subalgebras then we will denote by $B\vee C$ the von Neumann algebra generated by $B\cup C$  in $M$.   

Finally, whenever $\G$ is a group and $K,H\subseteq \G$ are subsets we will be denoting by $KH=\{kh\,:\, k\in K, h\in H\}$ and by $\langle K\rangle$ the subgroup generated by $K$ in $\G$.


\section{Some Preliminaries on Intertwining Results}\label{sec: prelim}

\subsection{Popa's intertwining techniques} Over a  decade ago Popa developed a powerful technology for conjugating subalgebras of tracial von Neumann algebras, now termed the  {\it intertwining-by-bimodules techniques}, \cite [Theorem 2.1 and Corollary 2.3]{Po03}. For further reference we recall the following theorem.

\begin {theorem}[Popa, \cite{Po03}]\label{corner} Let $(M,\tau)$ be a separable tracial von Neumann algebra and $P,Q$ be two (not necessarily unital) von Neumann subalgebras of $M$. 

Then the following are equivalent:

\begin{enumerate}

\item There exist  non-zero projections $p\in P, q\in Q$, a $*$-homomorphism $\theta:pPp\rightarrow qQq$  and a non-zero partial isometry $v\in qMp$ such that $\theta(x)v=vx$, for all $x\in pPp$.

\item There is no sequence $u_n\in\mathcal U(P)$ satisfying $\|E_Q(xu_ny)\|_2\rightarrow 0$, for all $x,y\in M$.
\end{enumerate}

\end{theorem}
If one of the two equivalent conditions in the theorem above holds then we say that {\it a corner of $P$ embeds into $Q$ inside $M$}, and write $P\preceq_{M}Q$. If in addition we have that $Pp'\preceq_{M}Q$, for any non-zero projection  $p'\in P'\cap 1_PM1_P$, then we write $P\preceq_{M}^{s}Q$.

\subsection{Finite index inclusions of tracial von Neumann algebras} If $P\subseteq M$ are II$_1$ factors, then the {\it Jones index} of the inclusion $P\subseteq M$, denoted  $[M:P]$, is  the dimension of $L^2(M)$ as a left  $P$-module. M. Pimsner and S. Popa showed that the number $[M:P]$ can be interpreted as the best constant appearing in several inequalities involving the conditional 
expectation $E_P$ \cite[Theorem 2.2]{PP86}. It also follows from their work that these constants can be used to define a ``probabilistic'' index of any inclusion of tracial von Neumann algebras \cite[Remark 2.4]{PP86}.  

\begin{definition}[Pimsner $\&$ Popa, \cite{PP86}] \label{index}
Let $(M,\tau)$ be a tracial von Neumann algebra  with a von Neumann subalgebra $P$. Let $$\lambda=\inf\;\{\|E_P(x)\|_2^2/\|x\|_2^2\;:\; x\in M_{+}\}.$$
The {\it index of the inclusion $P\subseteq M$} is defined as $[M:P]=\lambda^{-1}$, under the convention that $\frac{1}{0}=\infty$.
\end{definition}

 \noindent For further use we note the following basic facts: 
\begin{lem} \cite[Lemma 2.3]{PP86}\label{ramen} Let $(M,\tau)$ be a tracial von Neumann algebra and $P\subseteq M$ be a von Neumann subalgebra such that $[M:P]<\infty$. Then the following hold:
\begin{enumerate}
\item for every projection $p\in P$ we have $[pMp:pPp]<\infty$;
\item $M\preceq_{M}^{s}P$.
\end{enumerate}
\end{lem}

As explained in \cite{CIK13}, it turns out that this precise notion of  index is a well suited technical tool to study global decomposition properties for von Neumann algebras, up to intertwining. For instance, it enables one to show the following version of  \cite[Proposition 3.6]{CIK13} involving commuting subalgebras rather than masa's. Its proof is similar with the one presented in \cite{CIK13} but we include all details for reader's convenience.  

\begin{prop}\label{masa}
Let $(M,\tau)$ be a tracial von Neumann algebra and let $z\in M$ be a non-zero projection. Assume that   $P\subseteq zMz$ and $N\subseteq M$ are von Neumann subalgebras such that $P\vee (P'\cap zMz)\subseteq zMz$ has finite index and that $P\preceq_{M}N$.

Then one can find a scalar $s>0$, non-zero projections $r\in N, p\in P$, a subalgebra $P_0\subseteq rNr$, and a $*$-isomorphism $\theta: pPp\ra P_0$ such that the following properties are satisfied:
\begin{enumerate} 
\item  $P_0\vee (P_0'\cap rNr)\subseteq rNr$ has finite index;
\item there exist a non-zero partial isometry $v\in M$ such that $rE_N(vv^*)=E_N(vv^*)r\geq sr$ and $\theta(pPp)v=P_0 v=r v pPp$;
\item $E_N(v (pP'p\cap pMp )v^*)''\subseteq P_0'\cap rNr$.
\end{enumerate}
\end{prop}

\begin{proof} Since $P\preceq_{M}N$, one can find nonzero projections $p\in P$ and $q\in N$, a nonzero partial isometry $v\in pMq$, and a $*$-homomorphism $\theta: pPp\ra qNq$ such that $\theta(x)v=vx$, for all $x\in pPp$.   Notice that $v^*v\in pPp'\cap pMp$ and $q':=vv^*\in \theta (pPp)'\cap qMq$. Moreover, without any loss of generality, we can assume that the support projection of $E_N(q')$ equals $q$.  Observe that for every $x\in pPp$ and $y\in pPp'\cap pMp$ and  we have 
\begin{equation*}\begin{split} &\theta(x) vyv^*=vxyv^*=vyxv^*=vyv^*\theta(x). \end{split}\end{equation*} 
Thus we have $v(pPp'\cap pMp)v^*\subseteq  \theta(pPp)'\cap qMq $ and hence \begin{equation}\label{2.4.3}E_N (v(pPp'\cap pMp)v^*)\subseteq \theta(pPp)'\cap qNq.\end{equation}

Since  the inclusion $P\vee (P'\cap zMz)\subseteq zMz$ has finite index then also $pPp \vee (pPp'\cap pMp)=p(P\vee (P'\cap M))p\subseteq pMp$ has finite index and hence $v(pPp\vee (pPp'\cap pMp))v^*\subseteq vpMpv^*=q'Mq'$ has finite index too. 

For every $s>0$ we denote by $q_s=1_{[s,\infty)}(E_N(q'))$ and notice that $\|q_s-q\|\ra 0$, as $s\ra 0$ and $q_sE_N(q')=E_N(q')q_s\geq s q_s$. This further implies that $\|q_sv-v\|\ra 0$, as $s\ra 0$; in particular, we can pick $s>0$ such that $q_sv\neq 0$. Applying \cite[Lemma 2.3]{CIK13} (see also \cite[Lemma 1.6(1)]{Io11}) it follows that the  inclusion 
\begin{equation}\label{2.4.1} 
q_sE_N(v(pPp\vee (pPp'\cap pMp))v^*)'' q_s \subseteq q_sNq_s
\end{equation}
has finite index.
Since $v^*v\in pPp'\cap pMp$ then $v(pPp\vee (pPp'\cap pMp))v^*\subseteq  v(pPp)v^* \vee v(pPp'\cap pMp)v^*= \theta(pPp)vv^* \vee v(pPp'\cap pMp)v^*$. This further implies  \begin{equation*}\begin{split}E_N(v(pPp\vee (pPp'\cap pMp))v^*)'' &\subseteq E_N(\theta(pPp)vv^* \vee v(pPp'\cap pMp)v^*)'' \\ &\subseteq \theta(pPp)\vee   E_N(v(pPp'\cap pMp)v^*)''.\end{split}\end{equation*}  

This last containment together with relation (\ref{2.4.1}) give that the inclusion $\theta(pPp)q_s\vee   q_sE_N(v(pPp'\cap pMp)v^*)'' q_s \subseteq q_s Nq_s$ has finite index.  Then the statement follows from this and (\ref{2.4.3}) by letting $r:=q_s$ , and $P_0:=\theta (pPp) q_s \subset q_s Nq_s$. \end{proof}

\subsection{Dichotomy for normalizers inside crossed products}\label{sec: dichotomy}

Recently, Sorin Popa and Stefaan Vaes obtained a series of ground breaking results regarding the classification of normalizers of algebras inside crossed products arising from large families of groups including free groups \cite[Theorem 1.6]{PV11} or hyperbolic groups \cite[Theorem 1.4]{PV12}. Motivated by these results and by the remarkable subsequent developments in the realm of amalgamated free products due to Adrian Ioana \cite{Io12} (see also \cite{Va13}) we will introduce a new class of groups. However to be able to state it  properly we need one more definition.

\begin{definition} \cite[Section 2.2]{OP07}
Let $(M,\tau)$ be a tracial von Neumann algebra, $p\in M$ a projection, and $P\subset pMp,Q\subset M$ von Neumann subalgebras. We say that $P$ is {\it amenable relative to $Q$ inside $M$} if there exists a $P$-central state $\phi:p\langle M,e_Q\rangle p\rightarrow\mathbb C$ such that $\phi(x)=\tau(x)$, for all $x\in pMp$.
\end{definition}

\begin{definition} A group $\G$ belongs to class $\Cal C_{rss}(\Sigma)$ if it is exact \cite{KW99} and there exists a malnormal, proper subgroup $\Sigma<\G$ for which  the following dichotomy property holds:  Assume  $\Gamma\curvearrowright B$ is any trace preserving action on a tracial von Neumann algebra $(B,\tau)$ and denote by $M=B\rtimes\Gamma$. Let $p\in M$ be a projection and $A\subset pMp$ a von Neumann subalgebra that is amenable relative to $B\rtimes \Sigma$ inside $M$. Then either $A\preceq_{M}B\rtimes \Sigma $ or $P:=\mathcal N_{pMp}(A)''$ is amenable relative to $B\rtimes \Sigma$ inside $M$.\end{definition} 

Summarizing the results described above, the following classes of groups belong to $\Cal C_{rss}(\Sigma)$: 

\begin{enumerate}
\item \cite[Theorem 3.1, Lemma 4.1, and Theorem 7.1]{PV11} Any weakly amenable group with positive first $\ell^2$-Betti number---here $\Sigma=\langle e\rangle$;  
\item \cite[Theorem 3.1]{PV12} Any weakly amenable, non-amenable, bi-exact group---here $\Sigma=\langle e\rangle$;
\item \cite[Theorem 1.6]{Io12},\cite[Theorem A]{Va13} Any non-amenable, nontrivial free product $\G= \G_1\ast_\Lambda \G_2$, where $\G_i$ are exact and $\La<\G_i$ is assumed malnormal---here $\Sigma=\Lambda$.  
\end{enumerate}
From now on, whenever $\Sigma =\langle e \rangle$  then the class $\Cal C_{rss}(\Sigma)$ will be simply denoted by $\Cal C_{rss}$. In the same spirit as in \cite{VV14} one can establish that the class $\Cal C_{rss}$ is closed under commensurability.


\section{Class $Quot(\Cal C)$}\label{sec: quot} 

Let $\Cal C$ be a class of groups. We define $Quot_{1}(\mathcal C)=\mathcal C$. Given an integer $n\geq 2$, we say that a group $\G$ belongs to the class $Quot_n(\Cal C)$ if the following are satisfied: 
\begin{enumerate}
\item there exist a collection of groups  $\Gamma_k$, $1\leq k\leq n$ and a collection of surjective homomorphisms $ \pi_k:\G_{k}\rightarrow \G_{k-1}$ such that $\G_1\in \Cal C$ and $ker(\pi_k)\in \Cal C$, for all $2\leq k\leq n$;  
\item $\G$ and $\G_n$ are commensurable.
\end{enumerate}

\begin{definition} We denote by $Quot(\mathcal C):=\cup_{n\in\mathbb N} Quot_n(\mathcal C)$ and any group $\G\in Quot(\mathcal C)$ is called a \emph{finite-step extension of groups in $\mathcal C$}.
\end{definition}

Below we list some useful basic algebraic properties of this family of groups. We will omit most of the proofs as they are either straightforward or already contained in \cite[Lemmas 2.9-2.10]{CIK13}.

\begin{prop}\label{quot} The following properties hold: 
\begin{enumerate}
\item If $\G\in Quot_n(\Cal C)$ and $p: \La \ra \G$ is a surjective homomorphism such that $ker(p)\in\Cal C$ then $\La \in Quot_{n+1}(\Cal C)$. 
\item If $\G_i \in Quot_{n_i}(\mathcal C)$ for all $1\leq i\leq k$ then $\G_1 \times \G_2\times \cdots \times \G_k \in Quot_{n_1+n_2+\cdots +n_k}(\Cal C)$.
\item If $\Cal C$ is closed under commensurability then $Quot_n(\Cal C)$ is also closed under commensurability.
\item Let $\G\in Quot_n(\Cal C)$ and let $ \pi_k:\G_{k}\rightarrow \G_{k-1}$ be the surjective homomorphisms satisfying the previous definition and let $p_n=\pi_2\circ\pi_3\circ\cdots \circ\pi_n:\Gamma_n\rightarrow\Gamma_1$. Then the following hold:
\begin{enumerate} 
\item If $\La<\G_1$ is a subgroup such that $\La\in \Cal C$ then $p_n^{-1}(\La)\in Quot_n(\Cal C)$;
\item $\ker(p_n)\in Quot_{n-1}(\Cal C)$; moreover, if $\Cal C$ is closed under commensurability and $\La<\G$ is a subgroup such that $p_n(\La)$ is finite  then $p^{-1}_n(p_n(\La))\in Quot_{n-1}(\Cal C)$. 
\end{enumerate}

\item If $\Cal C$ is closed under commensurability up to finite kernel then $Quot_n(\Cal C)$ is also closed under commensurability up to finite kernel. 
\item If all the groups in $\Cal C$ are exact then so are all the groups in $Quot_n(\Cal C)$, \cite{DL14}. 
\end{enumerate}
\end{prop}

Before we present the proof we explain the terminology used in (5) above; two groups $\G_1$ and $\G_2$ are called \emph{commensurable up to finite kernel}  if there exist finite index subgroups $\La_i\leqslant \G_i$, a group $H$, and surjections $\psi_i:\La_i\ra H$ such that $ker(\psi_i)$ are finite. One can easily check that this is the smallest equivalence relation whose classes are closed under taking both finite index subgroups and quotients under normal finite subgroups. 
\vskip 0.08in
\noindent \emph{Proof of Proposition \ref{quot}} We show only (5). Assume that $\La_1$ is commensurable up to finite kernel with $\La_2$ and $\La_1\in Quot_n(\Cal C)$. Thus one can find  finite index subgroups $\Sigma_i <\La_i$, a group $H$, and surjective homomorphisms $\phi_i :\Sigma_i\ra H$ with finite kernels $\Omega_i:=ker(\phi_i)$. Denote by $\tilde\phi_i : \Sigma_i/\Omega_i\ra H$ the induced isomorphisms. Since $\La_1\in Quot_n(\Cal C)$, $\Sigma_1 < \La_1$ is finite index, and $\Cal C$ is closed under taking finite index subgroups,  it follows that $\Sigma_1\in Quot_n(\Cal C)$.  Thus there exist surjections $\pi_k :\G_k \ra \G_{k-1}$ such that $ker(\pi_k) \in \Cal C$ for every $2\leq k \leq n$, $\G_1 \in\Cal C$, and $\Sigma_1=\G_n$.  Let $\Theta_n:=\Omega_1$  and then for every $2\leq k\leq n$ define recursively $\Theta_{k-1}:= \pi_{k}(\Theta_k)< \G_{k-1}$; since $ \pi_k $'s  are surjections then $\Theta_{k-1}$ is a  finite normal subgroup of $\G_{k-1}$. Moreover, for each $2\leq k\leq n$ the map $\tilde \pi_k : \G_k/\Theta_k \ra \G_{k-1}/\Theta_{k-1}$ given by $\tilde \pi_k (x \Theta_k)= \pi_k(x)\Theta_{k-1}$ is a surjective homomorphisms and $ker(\tilde \pi_k)=  ker(\pi_k)\Theta_k$. By the isomorphism theorem we have $ker(\tilde \pi_k)\cong ker(\pi_k)/ (ker(\pi_k)\cap \Theta_k)$ and since $\Theta_k$ is finite and $\Cal C$ is closed under commensurability by finite kernel it follows that $ker (\tilde \pi_k)\in \Cal C$, for all $2\leq k\leq n$ and $\G_1/\Theta_1 \in \Cal C$. 
 If $p: \Sigma_2\ra \Sigma_2/\Omega_2$ denotes the canonical projection then the formulas $\tilde\pi'_n:= \tilde \pi_n\circ \tilde \phi^{-1}_2\circ \tilde\phi_1\circ  p: \Sigma_2  \ra \G_{n-1}/\Theta_{n-1}$ and $\tilde\pi'_{k}:=\tilde\pi_k: \G_k/\Theta_k \ra \G_{k-1}/\Theta_{k-1}$ for every $2\leq k\leq n-1$ define surjective homomorphisms. Moreover, since the kernel $ker(\tilde \pi'_n)= p^{-1}( \tilde \phi^{-1}_1 (\tilde \phi_2(ker(\pi_n))))$ satisfies $ker(\tilde \pi'_n)/\Omega_2 \cong ker(\pi_k)$ and $\Omega_2$ is finite we have $ker(\tilde \pi'_n)\in \Cal C$. By above we have $\ker(\tilde\pi'_k)\in \Cal C$, for all $2\leq k\leq n-1$, and $\G_1/\Theta_1\in \Cal C$ which further imply that $\Sigma_2 \in Quot_n (\Cal C)$. Finally, since from construction $\Sigma_2$ is a  finite index subgroup of  $\La_2$ we conclude that $\La_2\in Quot_n(\Cal C )$. $\hfill\square$
\vskip 0.04in 

As a common generalization for poly-cyclic, poly-free \cite{Me84}, and poly-hyperbolic groups \cite{DY08} we introduce the notion of poly-$\mathcal C$ groups, where $\mathcal C$ is a fixed class of groups. A group $\G$ is called \emph{poly-$\mathcal C$} if there exists a positive integer $n$ and a chain of groups $\langle e\rangle =\La_0\lhd \La_1\lhd \cdots \lhd \La_{n-1}\lhd \La_n=\G$ such that the quotients $\La_i/\La_{i-1}\in \mathcal C$, for all $0\leq i\leq n-1$. Below we show the poly-$\mathcal C$ groups are closely related with the family $Quot(\mathcal C)$ introduced at the beginning of the section.
\begin{prop} Given a group $\G$, the following conditions are equivalent: 
\begin{enumerate}
\item [a.] There exist a collection of groups  $\Gamma_k$, $1\leq k\leq n$ satisfying $\G=\G_n$ and $\G_1\in \Cal C$ and a collection of surjective homomorphisms $ \pi_k:\G_{k}\rightarrow \G_{k-1}$ satisfying  $ker(\pi_k)\in \Cal C$, for all $2\leq k\leq n$;  
\item [b.]$\G$ is a poly-$\mathcal C$ group.
\end{enumerate}

  \end{prop}  
  
  \begin{proof}
  First assume condition a.\ above is satisfied. For every integer $2\leq k \leq n$ consider the surjective homomorphism $p_k: \G_n\ra \G_{n-k+1}$ defined by $p_k=\pi_{n-k+2}\circ \pi_{n-k+3}\circ\cdots \circ \pi_{n-1}\circ\pi_n$. This family of homomorphisms naturally gives rise to following chain of normal subgroups: $\langle e\rangle\lhd \ker(p_2)\lhd \ker(p_3)\lhd \cdots \lhd \ker(p_{n-1})\lhd \ker(p_n)\lhd \G_n=\G$. Using the surjectivity of $\pi_k$'s and the isomorphism theorem we see that $\G_n/\ker(p_n)\cong \G_1\in \mathcal C$, $\ker(p_n)/\ker(p_{n-1})\cong \ker(\pi_2)\in \mathcal C$, $\ker(p_{n-1})/\ker(p_{n-2})\cong \ker(\pi_3)\in \mathcal C$, $\ldots$, $\ker(p_3)/\ker(p_{2})\cong \ker(\pi_{n-1})\in \mathcal C$, $ \ker(p_2)=\ker(\pi_n)\in \mathcal C$. This shows that $\G$ is poly-$\mathcal C$.
  
 To see the converse assume $\G$ is poly-$\mathcal C$. Hence one can find a chain of normal subgroups: $\langle e\rangle =\La_0\lhd \La_1\lhd \cdots \lhd \La_{n-1}\lhd \La_n=\G$ such that the quotients $\La_i/\La_{i-1}\in \mathcal C$, for all $0\leq i\leq n-1$. For every $2\leq k\leq n$ consider the surjective groups homomorphism $\pi_k: \La_n/\La_{n-k}\ra \La_n/\La_{n-k+1}$ defined by $\pi_k(x\La_{n-k})=x\La_{n-k+1}$,  for all $x\in \La_n$. From assumptions we have that $\La_n/\La_{n-1}\in \mathcal C$ and furthermore, using the isomorphism theorem, we see that $\ker(\pi_k)\cong (\La_n/\La_{n-k})/(\La_n/\La_{n-k+1})\cong \La_{n-k+1}/\La_{n-k}\in \Cal C$, for all $2\leq k\leq n$. Finally, denoting by $\G_1:=\La_n/\La_{n-1}$ and $\G_k:= \La_n/\La_{n-k}$, for all $2\leq k\leq n$, we see the conditions enumerated in part a.\ are satisfied. \end{proof}

\subsection{Relative hyperbolic groups} Let $\G$ be a group that is hyperbolic relative to $\Cal P=\{P_i \,:\, i\in I\}$ a finite family of infinite, proper, residually finite subgroups. Using very deep methods in geometric group theory Denis Osin was able to prove a powerful algebraic Dehn filling analog for this class of groups. As a consequence, it was shown in \cite[Theorem 1.1]{Os06} that there exist a non-elementary hyperbolic group $H$, a surjective homomorphism $\psi:\G\ra H$, and  finite index, normal subgroups $N_i\lhd P_i$ such that $ker (\psi)$ is the normal closure of $\cup_i N_i$ in $\G$, i.e., $ker(\psi)=\langle\langle \cup_iN_i\rangle \rangle^\G$. More recently, using the concept of very rotating family of subgroups, Dahmani, Guirardel, and Osin \cite[Theorem 7.9]{DGO11} were able to describe more concretely the structure of $ker(\psi)$, as a nontrivial free product. Precisely, they showed that there exist a family of nonempty subsets  $T_i \subset \G$ such that  $ker(\psi) = \ast_{i\in I} (\ast_{\g\in T_i} N_i^\g)$, where  $N_i^\g=\g N_i\g^{-1}$, (see also Osin's argument in \cite[Corollary 5.1]{CIK13}). Summarizing we have the following:
\begin{theorem}[\cite{Os06, DGO11}] Let $\mathcal H$ be the class consisting of all non-elementary hyperbolic groups and all  non-amenable, non-trivial free product of exact groups. Then for any group $\G$ that is hyperbolic relative to a finite family of residually finite, exact, infinite, proper subgroups we have $\G\in Quot_2(\Cal H)$. 
\end{theorem}

\subsection{Mapping class groups} Let $S_{g,k}$ be a connected, compact and orientable surface of genus $g$ with $k$ boundary components.
Throughout the paper, we assume that a surface satisfies these conditions unless otherwise mentioned.
Denote by $\mod(S_{g,k})$ the \textit{mapping class group} of $S_{g, k}$, i.e., the group of isotopy classes of orientation-preserving homeomorphisms from $S_{g,k}$ onto itself, where isotopy may move points of the boundary of $S_{g, k}$.
Using the Birman short exact sequence in combination with the earlier results of Birman and Hilden \cite{BH73} it was shown in \cite[Section 4.3]{CIK13} the following result. 

\begin{theorem}\label{thm-mcg}  If $\Cal F$ denotes the class of all groups commensurable with non-abelian free groups then we have the following:
\begin{enumerate}
\item [(i)] If $g=0$, $k\geq 4$ then $\mod(S_{g,k})\in Quot_{k-3}(\mathcal F)$; 
\item [(ii)]If $g=1$, $k\geq 1$ then $\mod(S_{g,k})\in Quot_{k}(\mathcal F)$;
\item [(iii)]If $g=2$, $k\geq 0$ then  $\mod(S_{g,k})\in Quot_{k+3}(\mathcal F)$.
\end{enumerate}
\end{theorem} It follows from \cite{FM11} that, for every positive integer $k$, the central quotient of the braid group with $k$ strands  $\tilde B_k$  can be canonically identified with a finite index subgroup of the mapping class group $\mod(S_{0,k+1})$. Moreover, notice that the central quotient of the pure braid group with $k$ strands $\tilde P_k$ is a normal subgroup of index $k!$ in $\tilde B_k$.  Combining with the above theorem
we have \begin{cor}\label{surjections} For every $k\geq 3$, we have that $\tilde B_k, \tilde P_k\in Quot_{k-2}(\Cal F)$.
\end{cor}

\subsection{Surface braid groups} Throughout this subsection, we let $M=S_{g, b}$ be a surface.
Consider $\pmod(M)$ the \textit{pure mapping class group} of $M$, i.e., the subgroup of $\mod(M)$ that consists of isotopy classes of orientation-preserving homeomorphisms from $M$ onto itself which also preserve each component of $\partial M$ as a set.
Then $\pmod(M)$ is a subgroup in $\mod(M)$ of index $b!$ .

Fix $k$ a positive integer. Denote by $F_k(M)$ the space of ordered $k$ distinct points of $M$ and let $PB_k(M)$ be the fundamental group of $F_k(M)$.
The group $PB_k(M)$ is called the \textit{pure braid group of $k$ strands} on $M$. From definitions we have the equality $PB_1(M)=\pi_1(M)$.
For the basic properties of the group $PB_k(M)$ we refer to \cite{B69a, B69b,B74, PR99}.

We fix $x_1,\ldots, x_k$, mutually distinct points of $M$. 
By \cite[Theorem 3]{FaNe62}, the map from $F_{k+1}(M)$ into $F_k(M)$ sending a point $(t_1,\ldots, t_{k+1})$ of $F_{k+1}(M)$ to $(t_1,\ldots, t_k)$ is a locally trivial fibration which has the fiber $F_1(M\setminus \{ x_1,\ldots, x_k\} )$.
Thus, we have the following exact sequence
\begin{multline}\label{fn-general}
\cdots \ra \pi_2(F_{k+1}(M))\ra \pi_2(F_k(M))\\ 
\ra \pi_1(M\setminus \{ x_1,\ldots, x_k\})\ra PB_{k+1}(M)\ra PB_k(M)\ra 1.
\end{multline}
If $(g, b)\neq (0, 0)$, then by \cite[Corollary 2.2]{FaNe62} we have $\pi_2(F_l(M))=0$, for any positive integer $l$,  and hence (\ref{fn-general}) gives the following short exact sequence
\begin{equation}\label{fn}
1\ra \pi_1(M\setminus \{ x_1,\ldots, x_k\})\ra PB_{k+1}(M)\ra PB_k(M)\ra 1.
\end{equation}

Following the terminology from \cite{PR99}, we say that $M$ is \textit{large} if the group $\pi_1(M)$ is non-elementarily hyperbolic. This is the case if and only if $(g, b)\neq (0, 0), (0, 1), (0, 2), (1, 0)$.

Let $S=S_{g, b+k}$ be a surface and choose $k$ components of $\partial S$. We suppose that $M$ is obtained by filling a disk to each of these $k$ components of $\partial S$.
Thus we get the following Birman exact sequence 
\begin{equation}\label{j-exact}
PB_k(M)\stackrel{j}{\ra} \pmod(S)\ra \pmod(M)\ra 1
\end{equation}
corresponding to the canonical embedding of $S$ into $M$ (\cite[Theorem 1]{B69b}, \cite[Theorem 9.1]{FM11} and \cite[Theorem 2.8.C]{Iv02}).
Denote by $Z_k(M)$ the center of $PB_k(M)$ and notice that by \cite[Corollary 1.2]{B69b} we have $\ker j<Z_k(M)$.
Denote by $\widetilde{PB}_k(M):=PB_k(M)/Z_k(M)$. When $M$ large then \cite[Proposition 1.6]{PR99} implies that the center $Z_k(M)$ is trivial and hence $j$ is an injective homomorphism and also $\widetilde{PB}_k(M)=PB_k(M)$.

\begin{Remark}\label{rem-center}
For any $M$ non-large, we will argue below that $\ker j=Z_k(M)$; in particular, the group $\tilde{PB}_k(M)$ can be identified with a normal subgroup of $\pmod(S)$ through $j$.

To justify our claim notice that when $g=0$ and $b\leq 2$ the group $\pmod(M)$ is trivial. Then using the exact sequence (\ref{j-exact}) we have that $PB_k(M)/\ker j\simeq \pmod(S)$. Since  by \cite[Section 3.4]{FM11} the center of $\pmod(S)$ is trivial, this further implies that $\ker j=Z_k(M)$; in this case we also get that $\widetilde{PB}_k(M)\simeq \pmod(S)$.

Now assume that $g=1$ and $b=0$.
In \cite[Corollary 1.3]{B69b}, two generators of $\ker j$ were described through the presentation of $PB_k(M)$ in \cite[Theorem 5]{B69a}, and it was shown that $\ker j$ is isomorphic to $\mathbb{Z}^2$. Combining this with \cite[Proposition 4.2]{PR99}, we obtain the equality $\ker j=Z_k(M)$.
\end{Remark}

\begin{theorem}\label{thm-sbg}
For every positive integer $k$, the following assertions hold:
\begin{enumerate}
\item[(i)] If $M$ is large and $b\geq 1$, then $PB_k(M)\in Quot_k(\Cal F)$;
\item[(ii)] If $M$ is large and $b=0$, then $PB_k(M)\in Quot_k(\Cal H)$;
\item[(iii)] If $g=0$, $b\leq 2$ and $b+k\geq 4$, then $\widetilde{PB}_k(M)\in Quot_{b+k-3}(\Cal F)$;
\item[(iv)] If $g=1$, $b=0$ and $k\geq 2$, then $\widetilde{PB}_k(M)\in Quot_{k-1}(\Cal F)$.
\end{enumerate}
\end{theorem}

\begin{proof}
Assertions (i) and (ii) follow inductively from the short exact sequence (\ref{fn}) and the equality $PB_1(M)=\pi_1(M)$.

If $g=0$, $b\leq 2$, and $b+k\geq 4$ then using the Remark \ref{rem-center} above we have that $\tilde{PB}_k(M)\simeq \pmod(S)$. Thus, assertion (iii) follows from Theorem \ref{thm-mcg} (i).

Suppose that $g=1$, $b=0$ and $k\geq 2$. Let $\rho \colon PB_{k+1}(M)\ra PB_k(M)$ be the surjection from the exact sequence (\ref{fn}).
As mentioned in Remark \ref{rem-center}, two generators of $Z_k(M)$ are described in \cite[Corollary 1.3]{B69b} and \cite[Proposition 4.2]{PR99}.
The homomorphism $\rho$ is induced by the map from $F_{k+1}(M)$ into $F_k(M)$ which sends a point $(t_1,\ldots, t_{k+1})$ of $F_{k+1}(M)$ to $(t_1,\ldots, t_k)$.
It follows from the description of generators of $Z_k(M)$ that $\rho(Z_{k+1}(M))=Z_k(M)$ which further gives rise to a surjection $\tilde{\rho}\colon \widetilde{PB}_{k+1}(M)\ra \widetilde{PB}_k(M)$.

Since $\pi_1(M\setminus \{ x_1,\ldots, x_k\})$ is isomorphic to the free group of rank $k+1$ and its center is trivial, then (\ref{fn}) induces in a canonical way a short exact sequence
\begin{equation}\label{exactcentralquot}1\ra \pi_1(M\setminus \{ x_1,\ldots, x_k\})\ra \widetilde{PB}_{k+1}(M)\ra \widetilde{PB}_k(M)\ra 1.\end{equation}
Since $PB_1(M)=\pi_1(M)\simeq \mathbb{Z}^2$ then the group $\widetilde{PB}_1(M)$ is trivial. Hence, assertion (iv) follows from (\ref{exactcentralquot}), by induction on $k$.
\end{proof}

\begin{Remark}
Next we briefly discuss the exceptional cases for $PB_k(M)$ not covered by Theorem \ref{thm-sbg}.
If $(g, b)=(0, 0)$ and $k\leq 3$, then the first theorem in \cite[Section VI.2]{FV62} implies that $PB_k(M)$ is finite.
If $(g, b)=(0, 1)$, then $PB_1(M)$ is trivial, and $PB_2(M)$ is the Artin pure braid group of two strands and thus is isomorphic to $\mathbb{Z}$ (\cite[Section 9.3]{FM11}).
If $(g, b)=(0, 2)$, then $PB_1(M)\simeq \mathbb{Z}$.
If $(g, b)=(1, 0)$, then $PB_1(M)\simeq \mathbb{Z}^2$.
\end{Remark}

\subsection {The Torelli group and the Johnson kernel}

Let $S=S_{g, k}$ be a surface. A simple closed curve in $S$ is called \textit{essential} in $S$ if it is neither homotopic to a single point of $S$ nor isotopic to a component of $\partial S$. When there is risk no confusion, by a curve in $S$ we mean either an essential simple closed curve in $S$ or its isotopy class. 
A curve $\alpha$ in $S$ is called \textit{separating} in $S$ if $S\setminus \alpha$ is not connected.
Otherwise $\alpha$ is called \textit{non-separating} in $S$. Whether $\alpha$ is separating in $S$ or not depends only on the isotopy class of $\alpha$.
A pair of non-separating curves in $S$, $\{ \beta, \gamma \}$, is called a \textit{bounding pair (BP)} in $S$ if $\beta$ and $\gamma$ are disjoint and non-isotopic and $S\setminus (\beta \cup \gamma)$ is not connected. This condition depends only on the isotopy classes of $\beta$ and $\gamma$.
Given a curve $\alpha$ in $S$, we denote by $t_\alpha \in \pmod(S)$ the \textit{Dehn twist} about $\alpha$.

We define the \textit{Torelli group} $\cali(S)$ to be the group generated by all elements of the form $t_\alpha$ and $t_\beta t_\gamma^{-1}$ with $\alpha$ a separating curve in $S$ and $\{ \beta, \gamma \}$ a BP in $S$.
We define the \textit{Johnson kernel} $\calk(S)$ as the group generated by all $t_\alpha$ with $\alpha$ a separating curve in $S$.
We refer the reader to \cite[Chapter 6]{FM11} for more background of these groups.

When $g\geq 2$ and $k\leq 1$, the Torelli group of $S$ is originally defined as the group of elements of $\mod(S)$ acting on $H_1(S, \mathbb{Z})$ trivially.
Using the results in \cite{Jo79} and \cite{Pow78}, this original group turns out to be equal to the group $\cali(S)$ defined above.
When $g=0$, any curve in $S$ is separating in $S$, and therefore we have $\cali(S)=\calk(S)=\pmod(S)$.

\begin{theorem}
The following assertions hold:
\begin{enumerate}
\item[(i)] If $g=1$ and $k\geq 2$, then $\cali(S), \calk(S)\in Quot_{k-1}(\Cal F)$;
\item[(ii)] If $g=2$ and $k\geq 0$, then $\cali(S), \calk(S)\in Quot_{k+1}(\Cal F)$.
\end{enumerate}
\end{theorem}

\begin{proof}
If $g=2$ and $k=0$, then the equality $\cali(S)=\calk(S)$ holds, and it is further isomorphic to the free group of infinite rank by results in \cite{Me92,BBM10}. The theorem for this case follows.

Suppose that either $g=1$ and $k\geq 2$ or $g=2$ and $k\geq 1$.
Choose a component of $\partial S$, and denote by $R$ the surface obtained by filling up a disk in this chosen component of $\partial S$.
Then by \cite[Theorem 4.6]{FM11} we have the Birman exact sequence
\begin{equation}\label{birman}
1\ra \pi_1(R)\stackrel{j}{\ra}\pmod(S)\stackrel{q}{\ra} \pmod(R)\ra 1
\end{equation}
corresponding to the canonical embedding of $S$ into $R$.
By \cite[Lemma 4.1.I]{Iv02}, through the injection $j$, each of standard generators of $\pi_1(R)$ induces either the element $t_\alpha$ with $\alpha$ a separating curve in $S$, its inverse, or the element $t_\beta t_\gamma^{-1}$ with $\{ \beta, \gamma \}$ a BP in $S$.
It follows that $j(\pi_1(R))<\cali(S)$. Also, by the definition of $\cali(S)$ we have $q(\cali(S))=\cali(R)$. Combining these with (\ref{birman}) we obtain the following short exact sequence
\begin{equation}\label{exacttor}1\ra \pi_1(R)\ra \cali(S)\ra \cali(R)\ra 1.\end{equation}
If $g=1$ and $k=2$, then $\cali(R)$ is trivial because there exist neither a separating curve in $R$ nor a BP in $R$.
Hence, when $g=1$ and $k\geq 2$ we get $\cali(S)\in Quot_{k-1}(\Cal F)$, by using (\ref{exacttor}) and induction on $k$.
Similarly, when $g=2$ and $k\geq 0$ we have $\cali(S)\in Quot_{k+1}(\Cal F)$, by induction on $k$, starting from the result in the first paragraph of the proof.

Suppose again that either $g=1$ and $k\geq 2$ or $g\geq 2$ and $k\geq 1$.
Restricting the the short exact sequence (\ref{birman}), we obtain the following short exact sequence
\begin{equation}\label{exactjohn}1\ra j(\pi_1(R))\cap \calk(S)\ra \calk(S)\ra \calk(R)\ra 1.\end{equation}
Denote by $N=j(\pi_1(R))\cap \calk(S)$. Through the injection $j$, any simple loop in $R$ surrounding exactly one component of $\partial R$ and cutting a cylinder from $R$ induces either the element $t_\alpha$ with $\alpha$ a separating curve in $S$ or its inverse; in particular, it follows that $N$ is infinite.
Also, relying on \cite{Jo80} it was proved in \cite[Proposition 2.2]{Ki09} that for any BP $\{ \beta, \gamma \}$ in $S$, no non-zero power of $t_\beta t_\gamma^{-1}$ lies in $\calk(S)$. This further implies that $N$ has infinite index in $j(\pi_1(R))$ and hence, by \cite[Theorem V.12.5]{DD89}, $N$ is a free group of infinite rank.
Finally,  we obtain  the conclusion of the theorem about $\calk(S)$, by using (\ref{exactjohn}) and induction on $k$.
\end{proof}


\section{Class $NC_1$}\label{sec: nc1}  

Let $\G$ be a countable discrete group and let $\pi: \G\ra  \Cal O( \Cal H)$ be an orthogonal representation. A map $q: \G\ra \Cal H$ is called a \emph{quasi-cocycle} for $\pi$ if there exists a constant $C\geq 0$ such that \begin{equation}\label{1}\sup_{\g,\la\in\G } \|q(\g\la)-q(\g)-\pi_\g(q(\la))\| \leq C.\end{equation}

The infimum of all constants  $C$ satisfying equation (\ref{1}) is called the \textit{defect} of quasi-cocycle $q$ and is denoted by $D(q)$. When the defect vanishes the quasi-cocycle $q$ is actually a $1$-cocycle with coefficients in $\pi$, \cite{BHV05}. Throughout the paper we will assume, without any loss of generality, that any  quasi-cocycle is \emph{anti-symmetric}, i.e., $q(\g)=-\pi_\g(q(\g^{-1}))$, for all $\g\in \G$. We can make this assumption because in fact every quasi-cocycle is within a bounded distance from an anti-symmetric quasi-cocycle,  \cite{Tho09}. The set of all unbounded, anti-symmetric quasi-cocycles for $\pi$ will be denoted by $\mathcal {QH}_{as}^1(\G, \pi)$. If $\Cal B_{as}(\G,\pi)$ denotes the set of all bounded, anti-symmetric maps  $b:\G \ra \mathcal H$ then we obviously have that  $\mathcal {QH}_{as}^1(\G, \pi)+ \Cal B_{as}(\G,\pi)=\mathcal {QH}_{as}^1(\G, \pi)$ and $(\mathbb R\setminus\{0\})\cdot\mathcal {QH}_{as}^1(\G, \pi)=\mathcal {QH}_{as}^1(\G, \pi)$.

\begin{definition} An orthogonal representation $\pi:\G \ra \Cal O(\Cal H)$ is called \emph{mixing} if for every $\xi,\eta \in \Cal H$ we have $\lim_{\g\ra \infty} \langle\pi_\g (\xi), \eta\rangle =0$.  
\end{definition}
Basic examples are any multiple of the (real) left regular representation,  $\oplus \ell_{\mathbb R}^2(\G)$.  The mixing property is preserved under many basic operations on representations including:  taking sub-representations, restrictions to infinite subgroups, direct sums, tensor products, and inductions to finite index supra-groups.

\begin{definition} A representation $\pi:\G \ra \Cal O(\Cal H)$ is called \emph{weakly}-$\ell^2$ if it is weakly contained in the left regular representation $\ell_{\mathbb R}^2(\G)$. 
\end{definition}
It follows from the definitions that any restriction of a weakly-$\ell^2$ representation to any of its subgroups is again weakly-$\ell^2$. Moreover, the weakly-$\ell^2$ property is preserved under direct sum and under induction to supragroups (this follows from a similar proof with \cite[Theorem F.3.5]{BHV05}). Thus, if  $\G$ is a group and $\{\Sigma_i \}$ is a countable family of amenable subgroups then by above it follows that the multiple $\oplus_i \ell^2_{\mathbb R}(\G/\Sigma_i)$ is weakly-$\ell^2$.

\begin{definition}[Notation 0.1 in \cite{CSU13}] We say that a group $\G$ belongs to class  $NC_1$ if it is non-amenable and there exists a weakly-$\ell^2$, mixing, orthogonal representation $\pi:\G \ra \Cal O(\Cal H)$ such that $\mathcal {QH}_{as}^1(\G, \pi)\neq \emptyset$.\end{definition}

Notice that this class includes the class $\Cal D_{reg}$ introduced by Thom in \cite{Tho09}. In the remaining part of this subsection we underline some basic properties of the class $NC_1$. Some of these have been already discussed in \cite[Section 1]{CSU13} but we will include them here to make the text more self contained. As we will see this class is quite rich, containing large families of groups which are intensively  studied in various areas  of mathematics, especially  topology, geometric group theory, or logic. In addition, we study permanence properties of $NC_1$ under various canonical constructions in group theory.  Many of these properties are either folklore or already appeared in the literature so many of their proofs will be skipped. 
\begin{prop}\label{NCsimple} The following properties hold: 
\begin{enumerate}
\item [a)] If $\Sigma<\G_1, \G_2$ are groups with $\Sigma$ finite and $[\G_1:\Sigma  ]\geq 2, [\G_2: \Sigma]\geq 3$ then $\G_1\ast_\Sigma \G_2 \in NC_1$,\cite{PV09,CP10};
\item [b)]  Given $\Sigma<\G$ groups with $\Sigma$ nontrivial finite, $\G$ infinite and $\Theta: \Sigma \ra \G$ is a monomorphism denote by $HNN(\G,\Sigma, \Theta)$ the corresponding HNN-extension; then $HNN(\G,\Sigma, \Theta)\in NC_1$, \cite{FV10, CP10};  
\item [c)] If a non-amenable group $\G$ acts on a tree with finite stabilizers on edges then $\G \in NC_1$; 
\item [d)]The class $ NC_1$ is closed under taking non-amenable normal subgroups;
\item [e)] If  $\G$ is either a direct product of infinite groups or admits a infinite normal amenable subgroup then $\G\notin NC_1$, \cite{MS04, Po08, Pe06, CS11} .
\end{enumerate}
\end{prop}

\begin{proof} We will only very briefly justify d).  Let $\G \in NC_1$  and let $\Sigma\lhd \G$ be any non-amenable, normal subgroup. Then there exists a weakly-$\ell^2$, mixing, orthogonal representation $\pi:\G \ra \Cal O(\Cal H)$  and $q \in \mathcal {QH}_{as}^1(\G, \pi)$.
From the previous observations it follows that the restriction $\pi_{|\Sigma} : \Sigma \ra \Cal O(\Cal H)$ is again weakly-$\ell^2$ and mixing. Moreover, since $\Sigma$ is normal in $\G$ then  \cite[Theorem 2.1]{CSU13} implies that the restriction $q_{|\Sigma}\in \mathcal {QH}_{as}^1(\Sigma, \pi_{| \Sigma})$, and hence $\Sigma\in NC_1$.
\end{proof}

\begin{lem}\label{NCfindex} The class $NC_1$ is closed under taking supragroups of finite index.
\end{lem}  

\begin{proof} Let $\Omega\in NC_1$ and let $\Omega<\G $ be finite index supragroup. By passing to a finite index subgroup of $\Omega$ we can assume without any loss of generality that $\Omega$ is normal in $\G$. Since $\G$ satisfies $NC_1$ there exists a mixing, weakly-$\ell^2$, orthogonal  representation $\pi: \Omega\ra \Cal O(\Cal H)$  and an unbounded quasi-cocycle $q:\Omega\ra \Cal H$. To get our conclusion we use a construction as in the proof of Kaloujnine-Krasner embedding theorem \cite{KK50}, or more commonly known as the induced representation of $\pi$ from $\Omega$ to $\Gamma$. 

Denote by $\hat \G = \G/\Omega= \{ \hat \la= \Omega \la \,:\, \la\in\G\}$ the (finite) quotient group. Fix $\{ t_{\hat \la} \,:\, \la\in\G\}\subset \G$ a complete set of representatives for the cosets $\Omega$ in $\G$ and notice that $\widehat {t_{\hat \la}}=\hat \la $, for every  $\la \in\G$; hence for every $\g,\la\in \G$, we have $f_\g(\hla):=t_{\hla} \g t^{-1}_{\wh{\la\g}}\in \Omega$.    

Now we define $\tilde \pi: \G \ra \Cal O (\oplus_{\hat\G}\Cal H)$ by letting $\tilde\pi_\g (\oplus_{\hla} \xi_{\hla})= \oplus_{\hla} ( \pi_{f_\g(\hla)}(\xi_{\wh{\la \g}}))$, for every $\g \in \G$ and $\oplus_{\hla} \xi_{\hla}\in \oplus_{\hat\G} \Cal H$.  It is straightforward exercise to check that  $\tilde\pi$ is an orthogonal mixing representation.

Also since $\pi$ is weakly-$\ell^2$ a similar argument with \cite[Theorem F.3.5]{BHV05}) shows $\tilde \pi$ is also weakly-$\ell^2$.

In the remaining part we show that the map $\tilde q: \G \ra \oplus_{\hat \G} \Cal H$ defined by $\tilde q (\g)=\oplus_{\hla} q(f_\g(\hla))$, for every $\g\in \G$, is an unbounded quasi-cocycle, which will conclude the proof.
To see this, fix $\g,\de\in \G$ and using the definitions together with the quasi-cocycle inequality for $q$ we get  
\begin{equation*}\begin{split}
\|\tilde q (\g\de) -\tilde\pi_\g(\tilde q(\de))- \tilde q(\g)\|^2&= \sum_{\hla} \| q(f_{\g\de}(\hla))- \pi_{f_\g(\hla)}(q (f_\de(\wh{\la\g})))- q(f_\g(\hla))\|^2 \\
&=\sum_{\hla} \| q(f_{\g}(\hla)  f_\de(\wh{\la\g}))- \pi_{f_\g(\hla)}(q (f_\de(\wh{\la\g})))- q(f_\g(\hla))\|^2 \\
&\leq \sum_{\hla} D^2(q)= |\hat \G | D^2(q).\end{split}
\end{equation*}
This computation shows that $\tilde q$ is a quasi-cocycle satisfying $D(\tilde q)\leq |\hat \G |^{1/2} D(q)$. 

Finally, picking the section map such that $t_{\hat e}=e$ one can see that $\|\tilde q(\g)\|\geq \|q(\g)\|$ and hence $\tilde q$ is unbounded.
\end{proof}

The following result parallels similar results for the class of acylindrically hyperbolic groups, \cite[Lemma 3.8]{MO13}.
 
\begin{theorem}\label{NCpreserve}  The class $NC_1$ is closed under commensurability. Moreover, every group of the form finite-by-$NC_1$ belongs to $NC_1$. 
\end{theorem}

\begin{proof}  Assuming $\G_i$ are groups such that  $\G_1 \in NC_1$ and $\G_1$ is commensurable with $\G_2$ we need to show that $\G_2\in NC_1$. From assumptions there exists a group $H$ such that $H<\G_1$, $H<\G_2$ and $[\G_1:H]<\infty$, $[\G_2:H]<\infty $.  Thus one can find a subgroup $H_0<H$ which is normal in $\G_1$ and satisfies $[H:H_0]\leq [\G_1:H_0]<\infty$. Since $\G_1 \in NC_1$, then by part d) in Proposition \ref{NCsimple} we have that $H_0 \in NC_1$. Since $[\G_2:H_0]=[\G_2:H][H:H_0]<\infty$ then Lemma \ref{NCfindex} further implies that $\G_2\in NC_1$, which finishes the proof of the first part of the statement.

For the second part, if $\G$ is a finite-by-$NC_1$ group there exists $\Omega\lhd \G$ a finite normal subgroup such that $\La=\G/\Omega \in NC_1$. Denote by $p: \G \ra \La$ the canonical projection. Thus there exist $\pi :\La \ra \mathcal O(\Cal H)$  a weakly-$\ell^2$, mixing, orthogonal representation and $q\in\mathcal {QH}^1(\La,\pi)$ an unbounded quasi-cocycle.  Notice that  for every $\g\in \G$ and $\xi\in \Cal H$ the formula $\tilde\pi _\g(\xi)=  \pi_{ p(\g)}(\xi)$ defines an orthogonal representation $\tilde \pi: \G\ra \Cal O(\Cal H)$.  Since $\pi$ is mixing and weakly-$\ell^2$ and $\Omega$ is finite it follows that $\tilde\pi$ is also mixing and weakly-$\ell^2$.
Also it is a straightforward exercise to check  that the map $\tilde q: \G\ra \mathcal H$ given by $\tilde q(\g)=q(p(\g))$ for every $\g \in \G$ defines an unbounded quasi-cocycle for $\tilde \pi$ whose defect satisfies $D(\tilde q)\leq D(q)$; thus $\G \in NC_1 $.  \end{proof}
   
The following result  will be used in the sequel.

\begin{lem}\label{2.4}Let $\G$ be a group and let $\pi:\G\ra \mathcal O (\mathcal H)$ be a weakly-$\ell^2$, mixing orthogonal representation such that  $ \mathcal {QH}^1(\G,\pi)\neq \emptyset$. If we fix $q\in \mathcal {QH}^1(\G,\pi)$ then for every infinite subgroup $\Lambda \subset \G$ we have that the centralizer $C_\G (\Lambda)$ is amenable or there exists $C\geq0$ such that $\langle \Lambda , C_\G (\Lambda)\rangle \subseteq B^q_C=\{ \la \in \G \,:\, \|q(\la)\|\leq C\}$. Here, for every subset $H\subseteq \G$, we have denoted by $C_\G(H)=\{ \g\in \G \,:\, \g h=h\g, \text{ for all } h\in H \}$.  
\end{lem}
\begin{proof} Assume that the centralizer $C_\G(\La)$  is non-amenable. Then proceeding as in the first part of the proof of \cite[Proposition 2.6]{CSU13} one can find $C_1\geq 0$ such that $\Lambda \subseteq B^q_{C_1}$.  Finally since every element of $C_{\G}(\Lambda)$
 normalizes the subgroup $\Lambda$, $\pi$ is mixing, and $\Lambda$ is infinite, then by the first part  in \cite[Theorem 2.1]{CSU13} one can find $C\geq 0$ such that  $\langle \Lambda ,C_\G(\Lambda)\rangle \subseteq B^q_C$.   \end{proof}
In the end of this subsection we describe some recent important progress in building quasi-cocycles through innovative methods in geometric group theory. Some of the first results emerged from the seminal work of Mineyev \cite{Mi01} and Mineyev, Monod, and Shalom \cite{MMS03}, who showed that every  Gromov hyperbolic group $\G$ admits an unbounded (even proper) quasi-cocycle into a finite multiple of its left-regular representation and hence it is in $\Cal D_{reg}$. This was generalized by Mineyev and Yaman to relatively hyperbolic groups, \cite{MY09}. Hamenst\"adt \cite{Ha} showed that all weakly acylindrical groups, in particular, non-elementary mapping class groups and $Out(\mathbb F_n)$, $n \geq 2$, belong to the class $\Cal D_{reg}$. More recently, Hull and Osin \cite{HO11} and independently Bestvina, Bromberg and Fujiwara \cite{BBF13} where able to find a unified approach to these results by showing that for every group which admits a non-degenerate, hyperbolically
embedded subgroup belongs to the class $\Cal D_{reg}$. Their key results are some beautiful extension theorems on quasi-cohomology. In fact, by very recent work of Osin \cite{Os13} the weak curvature conditions used in both papers, as
well as Hamenst\"adt's weak acylindricity condition, are equivalent to the notion of acylindrical
hyperbolicity formulated by Bowditch, cf. [op. cit.]. Collecting these results together, the following families of groups are known to be acylindrically hyperbolic. In particular, they belong to the class $\Cal D_{reg}$ and thus will be contained in class $NC_1$.
\begin{examples}\label{examplesNC}
The following classes of groups belong to $NC_1$:

\begin{itemize}
\item [a.] Gromov hyperbolic groups \cite{Mi01, MMS03};
\item [b.] Groups which are hyperbolic relative to a family of subgroups as in \cite{MY09,HO11};
\item [c.] The mapping class group $\mod(S_{g,k})$ of a surface $S_{g,k}$ with $3g+k-4\geq 0$, \cite{Ha}; 
\item [d.] $Out(\mathbb F_n)$, $n \geq  2$, \cite{Ha};
\item [e.] Non-virtually cyclic graph products $G\{\G_v\}_{v\in V}$ of non-trivial groups with respect to some finite irreducible graph $G$ with at least two vertices; in particular, non-virtually cyclic right angled artin groups which do not split as a products, \cite{MO13}. 
\end{itemize}
\end{examples}

\begin{prop}Let $M=S_{g, b}$ be a surface and let $k$ be a positive integer. Then, the following assertions hold:
\begin{enumerate}
\item[(i)] If $M$ is large, then $PB_k(M)\in NC_1$.
\item[(ii)] If either $g=0$, $b\leq 2$ and $b+k\geq 4$ or $g=1$, $b=0$ and $k\geq 2$, then $\widetilde{PB}_k(M)\in NC_1$.
\end{enumerate}
\end{prop}

\begin{proof}
Let $S=S_{g, b+k}$ be the surface in the exact sequence (\ref{j-exact}).
Suppose that $M$ is large. As mentioned right before Remark \ref{rem-center}, $PB_k(M)$ is isomorphic to a normal subgroup of $\pmod(S)$.
By Proposition \ref{NCsimple} d) and Examples \ref{examplesNC} c., we have $PB_k(M)\in NC_1$ which gives assertion (i).

Next, we  suppose the conditions in assertion (ii) are satisfied. By Remark \ref{rem-center}, we have that $\widetilde{PB}_k(M)$ is isomorphic to a normal subgroup of $\pmod(S)$.
Then Proposition \ref{NCsimple} d) and Examples \ref{examplesNC} c.\ again imply that $\widetilde{PB}_k(M)\in NC_1$,  which gives assertion (ii).
\end{proof}


\section{The class $NC_1\cap Quot(\mathcal C_{rss})$} 
Since all our main structural results are applicable to von Neumann algebras arising from groups belonging to $ NC_1\cap Quot(\mathcal C_{rss})$, it would be interesting to thoroughly investigate this class of groups. While a complete understanding of this class of groups remains an open problem for future study, following the previous two subsections, we know it includes all groups that are commensurable with the following concrete groups:  

\begin{enumerate}
\item Any infinite, central quotient of the pure braid group $PB_n(S_{g,k})$ of $n$ strands on a surface $S_{g,k}$---in particular, all surface pure braid groups $PB_n(S_{g,k})$, for $n\geq 1$ and either $g=1$ and $k\geq 1$ or $g\geq 2$ and $k\geq 0$;
\item Any mapping class group $\mod(S_{g,k})$, for $0\leq g\leq 2$ and $2g+k\geq 4$;
 \item Any Torelli group $\Cal I (S_{g, k})$ and Johnson kernel $\Cal K (S_{g, k})$, for $g=1, 2$ and $2g+k\geq 4$;
\item Any group that is hyperbolic relative to a finite family of exact, residually finite, infinite, proper subgroups. 
\end{enumerate}


\section{Gaussian Deformations Arising From Quasi-cocycles on Groups}\label{sec: deformations}

Throughout this section we will assume that  $\G$ is a countable group and $\pi:\G\ra \mathcal O (\mathcal H)$ is an orthogonal representation such that $\Cal{QH}_{as}^1(\G,\pi)\neq \emptyset$. Following \cite{PS09,Si10} to the orthogonal representation $\pi$ one can associate, via the Gaussian construction, a probability measure space $(Y_{\pi},\mu_\pi)$ and a family $\{\omega(\xi) \,:\,\xi\in \mathcal H\} $ of unitaries in $L^\infty(Y_\pi, \mu_\pi)$ such that $L^\infty (Y_\pi, \mu_\pi)$ is generated as a von Neumann algebra by the $\omega(\xi)$'s  and the following relations hold:
\begin{enumerate} \item $\omega(0)=1$, $\omega(\xi_1+\xi_2)=\omega(\xi_1)\omega (\xi_2) $,  $\omega(\xi)^*=\omega(-\xi)$, for all $\xi,\xi_1,\xi_2\in \mathcal H$; 
\item $\tau(\omega (\xi))=\exp(-\|\xi\|^2)$, where $\tau$ is the trace on $L^\infty(Y_{\pi})$ given by integration. \end{enumerate}
Furthermore, there is a p.m.p.\  action $\G\ca^{\hat \pi} (Y_\pi, \mu_\pi) $ called the \emph{Gaussian action associated to $\pi$} which in turn induces a trace preserving action  $\G \ca^{\hat \pi} L^\infty(Y_\pi, \mu_\pi)$ that satisfies $\hat\pi_\g(\omega(\xi) )=\omega (\pi_\g(\xi))$, for all $\g\in \G$ and $\xi\in\mathcal H$.  

Assume that $(N, \tau) $ is a finite von Neumann algebra endowed with a trace $\tau$,   $\G\ca^\sigma (N,\tau)$ is a trace preserving action and denote by $M=N\rtimes_\sigma \G$ the corresponding crossed product von Neumann algebra.  Then the \emph{Gaussian dilation} of $M$ is defined as the crossed product algebra $\tilde M=(N\bar \otimes L^\infty(Y_\pi,\mu_\pi))\rtimes_{\sigma\otimes \hat\pi} \G  $.

Fix $q\in \Cal{QH}_{as}^1(\G,\pi)$ an unbounded quasi-cocycle. Following \cite{Si10} (see also \cite{CS11}) we construct a deformation arising from $q$ through a canonical exponentiation procedure---throughout the text this will be referred to as the \emph{Gaussian deformation associated with $q$}. For every $t\in \mathbb R$ consider the unitary $V_t\in \Cal U ( L^2(N)\bar \otimes L^2(Y_\pi,\mu_\pi) \bar \otimes \ell^2 (\G ))$ defined by the formula
$$V_t(x\otimes y\otimes \delta_\g)= x\otimes \omega (tq(\g))y \otimes \delta_\g,$$
for every $x\in L^2(N)$, $y\in L^2(Y_\pi,\mu_\pi)$, and $\g\in \G$. In \cite{CS11} it was proved that $V_t$ is a strongly continuous one parameter group of unitaries also satisfying  the following transversality property, \cite{Po08}:

\begin{prop} \cite[Lemma 2.8]{CS11}\label{transversality} Under the previous assumptions, for each $t$ and any $\xi \in L^2(M)$, we have 
\begin{equation}2||e^\perp_M \cdot V_t(\xi)||_{2}^2 \geq ||\xi-V_t(\xi)||_{2}^2\geq ||e^\perp_M \cdot V_t(\xi)||_{2}^2 , \end{equation}
where $e_M$ denotes the orthogonal projection of $L^2(\tilde{M})$ onto $L^2(M)$ and $e^\perp_M=1-e_M$.
 \end{prop}

Notice also that the deformation  $V_t$ satisfies an ``asymptotic bimodularity'' property, a key notion to incorporate the ``bounded equivariance'' of quasi-cocycles into von Neumann algebra context. 

\begin{theorem}\cite[Lemma 2.6]{CS11}\label{almostbimodular}  For every $x,y \in N \rtimes_{\sigma, r} \G$ (the reduced crossed product) we have that 
\begin{equation}\lim_{t\ra 0}\left(\sup_{\xi \in (L^2( M)_1} \|xV_t(\xi)y-V_t(x\xi y)\|_2\right)=0. \end{equation}  
\end{theorem}      

In the remaining part of the section we prove a few other basic convergence properties for $V_t$ that will be of essential use in the sequel. Before we proceed to the concrete statements we notice that when $q$ is a $1$-cocycle (i.e. when $D(q)=0$) these properties follow easily and they are already used in one form or another throughout the literature, \cite{Va10}.       
\begin{lem}\label{ineq} There exists a function $f: \mathbb R \ra [0,2^{1/2}]$ satisfying $\lim_{t\ra 0}f(t) =0$ and such that for every $x,y\in M$ and $z\in N\rtimes_{alg} \G$ we have the following inequality:

\begin{equation}\label{3.2'''} \begin{split} & \max \{\|V_t(xy)-xV_t(y)\|_2, \|V_t(yx)-V_t(y)x\|_2\} \\
 &\leq 2\|x\|_\infty \|y-z\|_2+ \|z\|_\infty \left |sup(z)\right |^{1/2}\left (\|V_{2^{1/2} t}(x)-x\|_2 + f(t) \|x\|_2 \right).
\end{split}\end{equation}Here we denoted by $\left |sup(z)\right |$ the cardinality of the support of $z$ in $\G$.  
\end{lem}

\begin{proof} Using the triangle inequality and the fact that $V_t$ is a unitary we have that 
 \begin{equation}\label{3.1.1}\begin{split}\|V_t(xy)-xV_t(y)\|_2& \leq \|V_t(x(y-z))\|_2+\|xV_t(y-z)\|_2+\|V_t(xz)-xV_t(z)\|_2 \\
 & \leq 2\|x\|_\infty \|y-z\|_2+\|V_t(xz)-xV_t(z)\|_2.
 \end{split}\end{equation}

\noindent Next, let $z=\sum_\mu z_\mu u_\mu \in N\rtimes_{alg} \G $, with $z_\mu \in N$, and let $x= \sum_\la x_\la u_\la$, with $x_\la \in N$, be the Fourier decompositions of $z$ and $x$ respectively. Thus, using the Cauchy-Schwarz inequality together with the formula for $V_t$, we see that

\begin{equation}\label{3.1.2}\begin{split}
\|V_t(xz)-xV_t(z)\|_2 &= \|\sum_\mu V_t(xz_\mu u_\mu )-xV_t(z_\mu u_\mu)\|_2 \\
&\leq |\sup(z)|^{1/2} \left ( \sum_\mu \|V_t(xz_\mu u_\mu )-xV_t(z_\mu u_\mu)\|^2_2  \right )^{1/2}\\
& =|\sup(z)|^{1/2} \left ( \sum_{\mu,\la }\|x_\la \sigma_\la(z_\mu)\otimes \left (\omega (tq(\la \mu ))-\omega (t\pi_\la(q(\mu))  ) \right ) \|^2_2  \right )^{1/2}\\
& =|\sup(z)|^{1/2} \left ( \sum_{\mu,\la }(2-2 e^{ -t^2 \| q(\la \mu )-\pi_\la(q(\mu) ) \| ^2})\|x_\la \sigma_\la(z_\mu) \|^2_2  \right )^{1/2}\\
\end{split}\end{equation}   
Furthermore, using successively  the basic  inequalities $\|x_\la \sigma_\la(z_\mu) \|_2 \leq \|z_\mu \|_\infty \|x_\la\|_2\leq \|z \|_\infty \|x_\la\|_2$, $\| q(\la \mu )-\pi_\la(q(\mu) ) \| ^2\leq 2\|q(\la)\|^2+2D(q)^2$, and $e^{ - 2t^2 \| q(\la) \| ^2}\leq 1$  we see that the last term above is smaller than 

\begin{equation}\label{3.1.3}
\begin{split} 
&\leq \left |\sup(z)\right | \|z \|_\infty\left ( \sum_{\la}(2-2 e^{ - 2t^2 \| q(\la) \| ^2 -2t^2D(q)^2})\|x_\la \|^2_2  \right )^{1/2}\\
&\leq \left |\sup(z)\right | \|z \|_\infty\left ( \sum_{\la}(2-2 e^{ - 2t^2 \| q(\la) \| ^2})\|x_\la \|^2_2   + \sum_{\la}(2-2 e^{ - 2t^2 D(q)^2})\|x_\la \|^2_2  \right )^{1/2}\\
&= \left |\sup(z)\right | \|z \|_\infty\left (\| V_{2^{1/2} t}(x)- x\|^2_2   + \|x\|^2_2(2-2 e^{ - 2t^2 D(q)^2})  \right )^{1/2}\\
& \leq\left |\sup(z)\right | \|z \|_\infty\left (\| V_{ 2^{1/2} t}(x)- x\|_2   + \|x\|_2 (2-2 e^{ - 2t^2 D(q)^2})^{1/2} \right )
\end{split}
\end{equation}

So, letting $f(t)= (2-2 e^{ - 2t^2 D(q)^2})^{1/2}$, the inequalities (\ref{3.1.1}), (\ref{3.1.2}), and  (\ref{3.1.3}) give the first inequality in (\ref{3.2'''}). The second inequality in (\ref{3.2'''}) follows similarly and we leave the details to the reader.\end{proof}
 
 \begin{prop}\label{conv1} Let $X\subseteq (M)_1$ be a set  and let $F\subseteq M$ be a finite subset and denote by $cu(F)= \{\sum \la_i x_i \,:\, x_i\in F , \la_i \in \mathbb C, |\la_i|\leq 1 \}$. If $V_t\ra \text{Id}$ (or equivalently $e^{\perp}_M\cdot V_t\ra 0$ ) uniformly on $X$ then $V_t\ra \text{Id}$ (or equivalently $e^{\perp}_M\cdot V_t\ra 0$) uniformly on $cu(F)\cdot X\cdot cu(F)$.    
\end{prop}

\begin{proof} To show our statement, it is sufficient to prove that $V_t\ra \text{Id}$ uniformly on $X\cdot cu(F)$ and on $cu(F)\cdot X$. We will only show the former convergence as the later will follow in a similar manner. Moreover, using the triangle inequality it suffices to show that  $V_t\ra \text{Id}$ uniformly on $X\cdot cu(F)$ only when $F$ is a singleton, thus assume that $F=\{ y\}$. 

Fix $\ve>0$ and by Kaplansky's density theorem let $y_\ve \in N \rtimes_{alg} \G$ with $\|y_\ve\|_\infty\leq \|y\|_\infty$ such that 
\begin{equation}\label{3.3''''}\|y-y_\ve\|_2\leq \ve/6.\end{equation} 
Also, since $V_t\ra \text{Id}$ uniformly on $X\cup \{y\}$ and since $f(t)\ra 0$ as $t\ra 0$, one can find $t_\ve>0$ such that for all $|t|\leq t_\ve$ and all $x\in X$ we simultaneously have  
\begin{equation}\label{3.2''''} \begin{split}\| V_t(x)-x\|_2& \leq \ve/ \left ( 6 \|y_\ve\|_\infty \left |sup(y_\ve)\right |^{1/2}\right ); \\ 
\| V_t(y)-y\|_2& \leq \ve/3; and \\ 
f(t) &\leq \ve / \left (6 \|y_\ve\|_\infty \left |sup(y_e)\right |^{1/2}\right ).
\end{split}\end{equation}

\noindent Thus, applying the triangle inequality and then using the first inequality in (\ref{3.2'''}) in the previous lemma together with (\ref{3.2''''}), (\ref{3.3''''}), and $\|x\|_2\leq \|x\|_\infty \leq 1$ we see that for every $|t|\leq t_{\ve}/ 2^{1/2}$ we have

\begin{equation*}
\begin{split}
\|V_t(xy)-xy\|_2&\leq \|V_t(xy)-xV_t(y)\|_2+ \|V_t(y)-y\|_2\\
& \leq 2\|y-y_\ve\|_2+ \|y_\ve\|_\infty \left |sup(y_\ve)\right |^{1/2}\left (\|V_{2^{1/2} t}(x)-x\|_2 + f(t) \right)+ \|V_t(y)-y\|_2\\
&\leq \ve/3 +   \ve/6+\ve/6 + \ve/3=\ve, \end{split}
\end{equation*}
which finishes the proof. \end{proof}

\begin{cor}\label{conv'}  Let $X\subseteq (M)_1$ be a subset and for each $i=1,2$ let $a_i\in M_+$ be positive elements, $f_i \in M$ be projections, and $\la_i>0$ be scalars satisfying $f_i a_i=a_i f_i\geq \la_i f_i$. If $e^{\perp}_M\cdot V_t \ra 0$ uniformly on $a_1Xa_2$ then $e^{\perp}_M\cdot V_t \ra 0$ uniformly on $f_1Xf_2$. \end{cor}
\begin{proof} Since $f_i a_i=a_i f_i\geq \la_i f_i$ then one can find $x_i\in M$ such that $a_i x_i=f_i$. The statement follows then from the transversality property (Proposition \ref{transversality}) and Proposition \ref{conv1}. \end{proof}

\begin{Remark}\label{s-array} More generally, instead of being a quasi-cocycle assume that $q$ is an \emph{$s$-array} on $\G$, i.e., there exists a linear function $\psi: \mathbb R_+\ra \mathbb R_+$ such that $\| q(\la \g)-\pi_\la(q(\g))\|\leq \psi(\|q(\la)\|)$, for all $\la,\g\in\G$. Then, one can easily check that a version of Lemma \ref{ineq} still holds, with some constants in the inequality (\ref{3.2'''}) and the formula for $f$ there slightly modified. Consequently, Proposition \ref{conv1} and Corollary \ref{conv'} will also hold in this case.    
\end{Remark}

\begin{lem}\label{decay0} Let  $\Sigma <\G$ be  a subgroup and let $a \in N\rtimes \G$ be an element satisfying $0\leq a\leq 1$ and put $P=N\rtimes \G$. Then 
for every $\e>0$ there exists $t_\ve>0$ such that for all $|t|\leq t_\ve$ and all $x\in (P)_1$ we have 
\begin{equation}\label{3.2'}
\| e_M^\perp \cdot V_t  ( a' x)\|^2_2\leq \| e_M^\perp \cdot V_t  ( a x)\|^2_2 + \ve, 
\end{equation}
where $a'=E_P(a)$.
\end{lem}
\begin{proof} As before consider the Gaussian dilations  $\tilde M = (N\bar\otimes L^\infty(X_\pi,\mu_\pi))\rtimes \G$ and by $\tilde P = (N\bar\otimes L^\infty(X_\pi,\mu_\pi)) \rtimes \Sigma$ and we notice the following commuting square condition $E_{\tilde P}\circ E_M= E_{M}\circ E_{\tilde P}=E_P$. 

Fix $\ve>0$. By Kaplansky density theorem there exists $a_\ve\in N\rtimes_{alg}\G $ with $\|a_\ve \|_\infty\leq 1$ such that $\|a-a_\ve\|_2\leq \ve/4$. 
Also, using Theorem \ref{almostbimodular}, there exists $t_\ve>0$ such that for all $x\in (P)_1$ and all $0\leq |t|\leq t_\ve$ we have 
\begin{equation}\label{3.0} \begin{split}
& \| e^\perp_M\cdot V_t(a_\ve x) - a_{\ve}e ^\perp_M\cdot V_t( x)\|_2\leq \ve/4, \text{ and} \\
&\| e^\perp_M\cdot V_t( a'_{\ve} x) - a'_{\ve}e ^\perp_M\cdot V_t( x)\|_2\leq \ve/4,
\end{split}\end{equation}   
where we have denoted by $a'_\ve=E_P(a_\ve)$. Thus, inequalities (\ref{3.0}) in combination with relations $e_{\tilde P}\cdot e_M^\perp \cdot V_t  (  x)=e_M^\perp \cdot V_t  (  x)$, for all $x \in P$, $e_{\tilde P} a_\ve e_{\tilde P}=E_P(a_\ve) e_{\tilde P}$, and the basic inequality $\|e_{\tilde P} (\xi)\|_2 \leq \|\xi\|_2$, for all $\xi\in L^2(\tilde M)$, show that for every $|t|\leq t_\ve$ we have 
\begin{equation}\label{3.1}\begin{split}
\| e_M^\perp \cdot V_t  ( a x)\|_2& \geq \| e_M^\perp \cdot V_t  ( a_\ve x)\|_2 - \ve/4\\
& \geq \| a_\ve e_M^\perp \cdot V_t  (  x)\|_2-\ve/2\\
& =\| a_\ve  \cdot e_{\tilde P}\cdot e_M^\perp \cdot V_t  (  x)\|_2 -\ve/2 \\
& \geq \| e_{\tilde P}\cdot a_\ve \cdot  e_{\tilde P}e_M^\perp \cdot V_t  (  x)\|_2-\ve/2 \\
& = \| E_{P}(a_\ve) e_M^\perp \cdot V_t  (  x)\|_2 -\ve/2\\
& \geq \| e_M^\perp \cdot V_t  (  E_{P}(a_\ve)  x)\|_2 - {3\ve}/4\\
& \geq \| e_M^\perp \cdot V_t  (  E_{P}(a)  x)\|_2 - \ve.
\end{split}\end{equation}
Since $\ve>0$ was arbitrary then inequality (\ref{3.2'}) follows by squaring (\ref{3.1}). 
\end{proof}

\begin{lem}\label{decay6}Let  $\Sigma <\G$ be  a subgroup and let $p\in N\rtimes \Sigma=:P$ be a nonzero projection. Assume for every $\ve>0$ there exists $t^1_\ve>0$ such that for all $|t|\leq t^1_\ve$ and all $x\in \left(pP p \right )_1$ we have 
\begin{equation}\label{8.3}
\| e_M^\perp \cdot V_t  (x)\|^2_2\leq  \ve. 
\end{equation}
Then  there exists a nonzero element $r\in \mathcal Z(P)$ with $0< r\leq 1$ such that for every $\ve>0$ one can find $t_\ve>0$ such that for all $|t|\leq t_\ve$ and all $y\in \left(P\right )_1$ we have  
\begin{equation}\label{8.3'}
\| e_M^\perp \cdot V_t  (y r )\|^2_2\leq \ve.
\end{equation} 
\end{lem}

 \begin{proof} First we claim that there exists $r' \in \mathcal Z(P)$ a projection such that $r'p\neq 0$ and for every $\ve_1>0$ there exists $t_{\ve_1} >0$ such that for all $|t|\leq t_{\ve_1}$, $y\in \left(P\right )_1$ we have  
\begin{equation}\label{8.3.1'}
\| e_M^\perp \cdot V_t  (y r'p )\|^2_2\leq  \ve_1.
\end{equation} 

To see this we use a standard convexity argument \cite{CP10}, and \cite{Va10}. Notice that the set  $\mathcal G:= \{ n u_\g\,:\, n\in \mathcal U(N),\g\in\Sigma\}$ forms a dense subgroup of $P$. Consider the closed convex hull $\mathcal K(p)= \overline{conv}^{ \|\cdot \|_2} \{ zpz^*\,:\,z\in \mathcal G \}$ and denote by $q \in K(p)$ the unique element of  minimal $\|\cdot \|_2$. Notice that since $\| zqz^*\|_2=\|q\|_2$ and $zqz^* \in \mathcal K(p)$ for every $z\in\mathcal G$ then by uniqueness we have that $ zqz^*=q$  for every $z\in \mathcal G$. Thus $q\in P\cap \mathcal G'= \mathcal Z(P)$ and since $ctr_P(\mathcal K(p))= ctr_P(p)$ we conclude that $q= ctr_P(p)\in \mathcal K(p)$. 

Fix $1\geq  \ve>0$. Thus from the definition of $\mathcal K(p)$ one can find a finite subset   $ \mathcal F_{\ve} \subset \mathcal G$ and $0<c_s, s\in\mathcal F_{\ve}$ with $\sum_{s\in \mathcal F_{\ve}}  c_s=1$ such that
 \begin{equation}\label{3.3.4'}
 \| q-\sum_{s\in\mathcal F_{\ve}}  c_s s  p s^* \|_2\leq \ve/8.
 \end{equation}  
 
 Moreover, by Proposition \ref{transversality} and Theorem \ref{almostbimodular} above there exists $t^1_\ve>0$ such that for all $|t|\leq t^1_\ve$, $s\in \mathcal F_{\ve}$, and $x\in (P)_1$ we have  
 \begin{equation}\label{3.3.5'}
 \begin{split}
 \| s e^\perp _M\cdot V_t ( p s^* x p )-e^\perp _M\cdot V_t ( s p s^* x p)\|_2\leq \ve/8. 
  \end{split}
 \end{equation}
 
 Also from (\ref{8.3}) there exists  $t^2_\ve>0$ such that for all $|t|\leq t^2_\ve$, $s\in \mathcal F_\ve$,  and $x\in (P)_1$ we have
 \begin{equation}\label{8.3'''}
\| e_M^\perp \cdot V_t  ( p s^* x p)\|^2_2\leq  \ve/8. 
\end{equation}

 Using the triangle inequality together with $\|V_t(\xi)\|_2\leq \|\xi\|_2$, for $\xi\in L^2(M)$, (\ref{3.3.4'}),  (\ref{3.3.5'}), (\ref{8.3'''}) and Cauchy-Schwarz inequality for every $x\in (P)_1$ and $|t|\leq \min\{ t^1_\ve,t^2_\ve\}$, we have 
 \begin{equation*}\begin{split}
\| e^\perp _M\cdot V_t ( q x p)\|^2_2 & \leq \left(\| e^\perp _M\cdot V_t ( (q- \sum_{s\in\mathcal F_{\ve}}  c_s s  p s^*) x p)\|_2 +\sum_{s\in\mathcal F_{\ve}} c_s\|e^\perp _M\cdot V_t (   s  p s^* x p) \|_2\right)^2\\
&\leq  (\| e^\perp _M\cdot V_t ( (q- \sum_{s\in\mathcal F_{\ve}}  c_s s  p s^*) x p)\|_2 +\\
&\quad + \sum_{s\in\mathcal F_{\ve}} c_s\|e^\perp _M\cdot V_t (   s  p s^* x p) - se^\perp _M\cdot V_t (    p s^* x p) \|_2 +\\
& \quad +  \sum_{s\in\mathcal F_{\ve}} c_s \|e^\perp _M\cdot V_t (   p s^* x  p ) \|_2)^2\\
& \leq \left(\ve/8+ \ve/8+ \sum_{s\in\mathcal F_{\ve}} c_s \|e^\perp _M\cdot V_t (   p s^* x  p ) \|_2\right)^2\\
& \leq \ve^2/16 + \ve/2+ \sum_{s\in \mathcal F_{\ve}}c_s  \|  e^\perp _M\cdot V_t ( p s^* x p)\|^2_2\\
& \leq \ve^2/16 +\ve/2 + \ve/8\leq \ve
\end{split}
\end{equation*}  
 Altogether, we have obtained that for every $\ve>0$ there exists $t_\ve>0$ such that for all $|t|\leq t_\ve$ and all $x\in(P)_1$ we have 
 \begin{equation}\label{8.3.11}
 \|e^\perp_M\cdot V_t(qxp)\|^2_2 \leq \ve.
   \end{equation}
   
 For every $\mu>0$ we denote by $q_\mu$ the spectral projection of $q$ corresponding to the interval $(\mu,\infty)$ and notice that $q_\mu \nearrow r:=supp(q)$ increasingly in SO- topology, as $\mu\searrow 0$. Thus there exists $\delta>0$ such that $q_\delta q\neq 0$ and since $q=ctr_P(p)$ it follows that $q_\delta p\neq 0$. Moreover, since $q_\delta q  \geq \delta q_\delta$ there exists an element $x_\delta \in \mathcal Z(P)$ such that $qq_\delta x_\delta =q_\delta$ and $\|x_\delta\|_\infty \leq \delta^{-1}$.
 
 Fix $\ve_1>0$. From (\ref{8.3.11}) there exists $t_{\ve_1}$ such that for all $|t|\leq t_{\ve_1}$ and all $x\in(P)_1$ we have
 
  \begin{equation*}\label{8.3.11'}
 \|e^\perp_M\cdot V_t(qxp)\|^2_2 \leq \ve_1 \delta^2.
   \end{equation*}
   If in this inequality we let $x= \delta q_\delta x_\delta y $ for arbitrary  $y \in (P)_1$ then we get our claim for $r'=q_\delta$.

  Finally, we notice that our claim together with same averaging argument as used in its proof further implies (\ref{8.3'}), where $r= ctr_P(r'p)$. In fact, the arguments presented above apply verbatim and we leave the details to the reader.  \end{proof} 

\begin{lem}\label{decay7} Assume the previous notations and let $N\rtimes \Sigma=:P$. Assume that there exists a finite subgroup $\Omega<\Sigma$ such that $\mathcal Z(P)\subseteq N\rtimes \Omega$. Also suppose there is a nonzero element $0\leq r\leq 1$ with $r\in\mathcal Z(P)$ such that for every $\ve>0$ there exists $t_\ve>0$ such that for all $|t|\leq t_\ve$ and all $x\in \left(P\right )_1$ we have 
\begin{equation}\label{3.3}
\| e_M^\perp \cdot V_t  (x r)\|^2_2\leq  \ve. 
\end{equation}
Then the quasi-cocycle $q$ is bounded on $\Sigma$.   
\end{lem}

\begin{proof} Fix $0<\ve<1$ and applying the assumption for $\ve\|r\|^2_2$ it follows that there exists $t> 0$ such that for all $\g\in \Sigma$ we have 
\begin{equation}\label{9.1.1}
\| e_M^\perp \cdot V_t  (u_\g r)\|^2_2\leq \ve \|r\|^2_2.
\end{equation}
Consider $r=\sum_{\omega \in\Omega}  r_\omega u_\omega$ be the its Fourier decomposition. Thus using the definition of $V_t$ we see that (\ref{9.1.1}) is equivalent to
\begin{equation*}\label{9.1.2}
\sum_{\omega\in\Omega}  (1-e^{-t^2\|q(\g\omega)\|^2}) \|r_\omega\|^2_2\leq\ve \sum_{\omega\in\Omega} \|r_\omega\|^2_2.
\end{equation*}
In particular, this inequality implies that for every $\g \in \Sigma$ there exists $\omega \in \Omega$ such that  $1-e^{-t^2\|q(\g\omega)\|^2}\leq  \ve$ or equivalently
$\|q(\g\omega)\|\leq (\ln(2/(1-\ve))^{1/2})/t$. Using the quasi-cocycle relation this, further entails that $\|q(\g)\|\leq D(q)+\|q(\omega)\| + (\ln(2/(1-\ve))^{1/2})/t$. Altogether, we have 
$\|q(\g)\|\leq D(q)+\sup_{\omega \in\Omega }\|q(\omega)\| + (\ln(2/(1-\ve))^{1/2})/t$, for all $\g\in \Sigma$, and since $\Omega$ is finite it follows that $q$ is bounded on $\Sigma$.  \end{proof}

\begin{Remark} Finally, we leave to the reader to check that all the previous Lemmas \ref{decay0}- \ref{decay7} still hold if instead of being a quasi-cocycle one assumes that $q$ is just an (anti)symmetric \cite{CS11} $s$-array on $\G$, as defined in Remark \ref{s-array}. Essentially, all the proofs will follow in the same way.  
\end{Remark}


\section{Applications to Group Theory}

The presence of a non-trivial quasi-cocycle on a group taking values in its left regular representation restricts significantly the internal structure of the group; for instance, it excludes most relations of ``order one'' like (asymptotic) commutation, etc. In the same spirit, we will show that any such group has at most finitely many finite conjugacy classes. More precisely, appealing to the representation theory techniques which steam mainly from \cite{CSU13} we show the following more general statement.  
 
 \begin{theorem}Let $\G$ be a non-amenable group together with $ \Sigma\lhd\G$, a normal non-amenable subgroup.  If $\G \in NC_1$ then there are only finitely many finite orbits for the action of $\Sigma$ on $\G$ by conjugation. \end{theorem}

\begin{proof} Since $\G \in NC_1$ there exists a weakly-$\ell^2$, mixing orthogonal representation $\pi:\G\ra \Cal O(\Cal H)$ and $q\in \Cal{QH}^1_{as}(\G,\pi)$. 

Let $\{ \mathcal O_n \,:\, n\in \mathbb N\}$ be an enumeration of all the (disjoint) finite orbits for the action of $\Sigma$ on $\G$ by conjugation. Notice that $\cup_{n\in\mathbb N} \Cal O_n= \{ \g\in\G \,:\,[\Sigma:C_\Sigma(\g)]<\infty\}=:\Lambda$, where $C_\G(\g)$ is the centralizer of $\g$ in $\G$. Since $\Sigma$ is normal in $\G$, it is a straightforward exercise to show that $\Lambda$ is a normal subgroup of $\G$. For every $n\in \mathbb N$ denote by $\Lambda_n:=\cup^n_{i=1}\Cal O_i$. 

The proof relies heavily on the techniques used in  \cite[Theorems 3.1 and 3.5]{CSU13} so we will only include a brief sketch on how to fit together these results. Denote by $M=L(\G)$ the corresponding group von Neumann algebra and let $u_\g$, with $\g\in\G$  be the canonical group unitaries. For every $n\in \mathbb N$ denote by $\xi_n=|\Lambda_n|^{-1/2}\sum_{a\in \Lambda_n} u_a  \in  M\subset L^2(M) $.  Then a basic calculation shows that for every $\g\in \Sigma$ and $n\in \mathbb N$ we have 
\begin{equation*}
u_\g  \xi_n =\xi_n u_\g. \end{equation*} 

Denote by $\tilde M =L^\infty(Y^\pi)\rtimes \G$ the Gaussian dillation associated with $\pi$. Let $V_t:L^2(M)\ra L^2(\tilde M)$, with $t\in \mathbb R$, be the Gaussian deformation corresponding to $q$ as defined in Section \ref{sec: deformations}. Since $\Sigma$ is non-amenable and $\pi$ is weakly-$\ell^2$ one can find a finite subset $E\subset \Sigma$ and $K\geq 0$ such that for every $\xi \in L^2(\tilde M )\ominus L^2(M)$ we have that
\begin{equation*} \sum_{\g\in E}  \|u_\g \xi-\xi u_\g \|_2\geq K\|\xi\|_2.
\end{equation*}
Then Proposition \ref{almostbimodular} above combined with the same spectral gap argument from the beginning of the proof of theorem \cite[Theorem 3.1]{CSU13} show that $$\lim_{t\ra 0}\left(\sup_n\|e^\perp_M \cdot V_t(\xi_n)\|_2\right)=0.$$ Thus the transversality property (Proposition \ref{transversality})  will further imply that
\begin{equation*}\lim_{t\ra 0}\left(\sup_n\|\xi_n-V_t (\xi_n)\|_2\right)=0.\end{equation*} Then a simple calculation shows that for every $\ve >0$ there exists $C\geq 0$ such that
\begin{equation}\label{ninner2}\sup_n\|\xi_n- P_{B'_C} (\xi_n)\|_2\leq \ve.\end{equation} As before, we have denoted by $P_{B'_C}$ the orthogonal projection from $\ell^2(\G)$ onto the Hilbert subspace $\ell^2(B'_C)$ with
$B'_C=\{
\la\in \G \,:\, \|q(\la)\|\leq C, \la\neq e\}$ being the ball of radius $C$ centered and pierced at the identity element $e$.

Then the same argument as in the proof of \cite[Theorem 3.5, pages 15-16]{CSU13} shows that for every $\ve>0$  there exists $C\geq 0$ such that for every $\g\in \Sigma$ we have
 \begin{equation}\label{boundedseq}\limsup_n\|P_{A_\g}(\xi_n)\|^2_2\geq 1-6\ve^2,\end{equation}
 where $A_\g= \g B'_C\g^{-1}\cap B'_C$. 
 Since $\Sigma$ is normal in $\G$ then by \cite[Theorem 2.1]{CSU13} the quasi-cocycle is unbounded on $\Sigma$.  Thus, since $\g$ does not depend on $\ve$ or $C$, one can pick $\g \in \Sigma \setminus B_{2C+2D(q)}$ and by \cite[Theorem 2.1]{CSU13} again it follows $A_\g$ is finite. Moreover, using the definition of $\xi_n$ we see that
 $ \|P_{A_\g}(\xi_n)\|^2_2=  |A_\g\cap \Lambda_n| |\Lambda_n|^{-1}$, for every $n$. This together with (\ref{boundedseq}) imply that there exists an integer $n_0$ such that $\La_k=\La_{n_0}$, for every $k\geq n_0$; hence $\La=\La_{n_0}$ is finite.
  \end{proof}
If we let $\Sigma=\G$ in the previous theorem we notice the following immediate corollary.  
 \begin{cor}Let $\G\in NC_1$. Then $\G$ has only finitely many finite conjugacy classes. Hence there exists a short exact sequence of groups
 $1\ra F \ra \G\ra \G_0\ra 1$, where $F$ is a finite and $\G_o$ is infinite conjugacy class. In particular, if $\G$ is assumed torsion free then $\G$ is infinite conjugacy class. \end{cor}

The following corollary is a straightforward consequence of the previous results. 

\begin{cor}\label{fcenter} Let $\G\in NC_1$ be a non-amenable group together with $\Sigma \lhd \G$ a non-amenable normal subgroup. Assume that $\G \ca N$ is a trace preserving action on a finite von Neumann algebra $N$. If $N\rtimes \G$ denotes the corresponding crossed product von Neumann algebra  then there exists $\La\lhd \G$ a finite normal subgroup such that $\Cal Z(N\rtimes \Sigma)\subseteq(N\rtimes \Sigma)'\cap (N\rtimes \G) \subseteq N\rtimes \La$.  \end{cor}


\section{Primeness Results for von Neumann Algebras of Groups in $NC_1\cap Quot(\mathcal C_{rss})$}

In this section we will use the technical results from the previous sections to derive the proof of Theorem \ref{main1}. Our arguments are similar in essence with the ones used in \cite{CIK13} but they have slightly different technical forms. For the sake of completeness we include all details. To simplify the writing in the main proof  we first introduce a notation:

\begin{notation} \label{notationres} Fixing a group $\Gamma_n \in Quot_n(\Cal C_{rss})$, there exist groups $\G_1, \G_2, \ldots, \G_{n-1}$ and a collection of  surjective homomorphisms $\pi_k:\G_{k}\rightarrow \G_{k-1}$ such that $\G_1\in \Cal C_{rss}$, and  $ker(\pi_k)\in \Cal C_{rss}$, for all $2\leq k\leq n$. Then we define $\theta_n=\pi_2\circ\pi_3\circ\cdots\circ\pi_n:\Gamma_n\rightarrow\Gamma_1$ and notice that by Proposition \ref{quot} we have that $ker (\theta_n)\in Quot_{n-1}(\Cal C_{rss})$. 
\end{notation}
\vskip 0.07in

\noindent {\bf  Proof of theorem \ref{main1}}.  Since both $NC_1$ and $Quot_n(C_{rss})$ are closed under commensurability then it will be sufficient to treat the case $\La/\Omega=\G_n\in NC_1 \cap  Quot_n(C_{rss})$, where $\Omega\lhd \Lambda$ is a finite normal subgroup. 

Throughout the proof we will denote by $\{u_\la\}_{\la\in \Lambda}\subset L(\Lambda)=:M$ and $\{v_\g\}_{\g\in\G_1}\subset L(\G_1)$ the canonical unitaries. We will prove our statement by induction on $n$. 

First we argue for $n=1$. Since $\G_1 \in \Cal C_{rss}$ denote by $\theta:\Lambda \ra \G_1$ the canonical projection. Consider the $*$-homomorphism $\tilde \theta: M\ra M\bar \otimes L(\G_1)$ given by $\tilde\theta(u_\la)=u_\la\otimes v_{\theta(\la)}$, for all  $\la\in\Lambda$. Assume by contradiction that $B,C\subseteq pMp$ are two commuting, diffuse subalgebras such that the inclusion $B\vee C \subseteq pMp$ has finite index. Hence $\tilde\theta(B)$, $\tilde\theta(C)$ are commuting, diffuse subalgebras of $\tilde\theta(p)\left (M\bar\otimes L(\G_1)\right )\tilde\theta(p)$. 
Let $P\subseteq B$ be an arbitrary diffuse, amenable subalgebra. Then $\tilde\theta(P)$ is amenable and hence it is amenable relatively to $M \otimes 1$ inside   $M\bar\otimes L(\G_1)$. Thus by the dichotomy property we either have that 

\begin{enumerate}
\item $\tilde\theta(P) \preceq_{M\bar\otimes L(\G_1)}  M\otimes1$, or 
\item $\tilde\theta(C)$ is amenable relative to $M \otimes 1$ inside   $M\bar\otimes L(\G_1)$. 
\end{enumerate}
Moreover, the case (2) above further implies, by the same dichotomy theorem, that either  \begin{enumerate} 
\item[(3)] $\tilde\theta(C) \preceq_{M\bar\otimes L(\G_1)}  M\otimes1$,  or  
\item[(4)] $\tilde\theta(B\vee C)$ is amenable relative to $M\otimes1$ inside $M\bar\otimes L(\G_1)$. 
\end{enumerate}

As in the proof of  \cite[Theorem 3.1]{CIK13} we show that case (4) above will lead to a contradiction. Indeed since $B\vee C\subseteq pMp$ has finite index then $pMp\preceq^s_{pMp} B\vee C$ and hence $pMp$ is amenable relative to $B\vee C$ inside $pMp$. This implies that $\tilde\theta(pMp)$ is  amenable relative to $\tilde \theta(B\vee C)$ inside $M\bar\otimes L(\G_1)$ and by \cite[Proposition 2.4] {OP07} it follows that $\tilde\theta(pMp)$ is amenable relative to $M\otimes 1$ inside $M\bar\otimes L(\G_1)$. Finally,  \cite[Proposition 3.5]{CIK13}  further implies that $\theta(\La)=\G_1$ is amenable which is a contradiction.

In conclusion, for every $P\subseteq B$ be an arbitrary diffuse, amenable subalgebra we have either (1) or (3) above. Due to symmetry, we can assume without any loss of generality that $\tilde\theta(B)\preceq_{M\bar\otimes L(\G_1)}M\otimes 1$. By \cite[Proposition 3.4]{CIK13} this further implies that $B\preceq_M L(\Omega)$ and since $\Omega$ is finite it follows that $B$ is not diffuse which contradicts the assumptions; this completely settles case $n=1$.

Next we show the inductive step. Since $\La/\Omega =\G_n \in \Cal C_{rss}$ there exists a surjection $\theta'=\theta_n \circ \theta : \La\ra \G_1$, where $\theta_n$ is the homomorphism from Notation \ref{notationres}. This allows us to define  a $*$-homomorphism $\tilde\theta':M\rightarrow M\bar{\otimes}L(\G_1)$ by letting $\tilde\theta'(u_\la)=u_\la \otimes v_{\theta' (\la)}$, for all $\la\in\La$. Assume by contradiction that $B,C\subseteq pMp$ are two commuting, diffuse subalgebras such that the inclusion $B\vee C \subseteq pMp$ has finite index. Thus $\tilde\theta'(B)$, $\tilde\theta'(C)$ are two commuting diffuse subalgebras of $\tilde\theta'(p)\left (M\bar\otimes L(\G_1)\right )\tilde\theta'(p)$. Proceeding as in $n=1$ case one of the following must hold:  
 \begin{enumerate}
\item [(5)] $\tilde\theta'(B) \preceq_{M\bar\otimes L(\G_1)}  M\otimes1$; 
\item[(6)] $\tilde\theta'(C) \preceq_{M\bar\otimes L(\G_1)}  M\otimes1$; 
\item[(7)] $\tilde\theta'(B\vee C)$ is amenable relative to $M\otimes1$ inside $M\bar\otimes L(\G_1)$. 
\end{enumerate}
 As in that proof, case (7) implies that $\theta'(\La)=\G_1$ is amenable which is a contradiction and cases (5) and (6) implies that  $B \preceq_M  L(ker(\theta')))$
and $C \preceq_M  L(ker(\theta'))$, respectively. Also, notice that if $B$ is amenable (and hence $C$ non-amenable!) we automatically have that $B \preceq_M  L(ker(\theta')))$. Thus, by the previous discussion,  it suffices to treat only the case  $B \preceq_M  L(\Sigma)$, where $ \Sigma=ker(\theta')$. By Proposition \ref{masa} one can find $s>0$, non-zero projections $r\in L(\Sigma), q\in B$, a subalgebra $B_o\subseteq rL(\Sigma)r$, and a $*$-isomorphism $\theta: qBq\ra B_o$ such that the following properties are satisfied:

\begin{eqnarray} 
&&\label{7.2} B_o\vee (B_o'\cap rL(\Sigma)r)\subset rL(\Sigma)r \text{ has finite index};\\
 &&\label{7.1} \text{ there exist a non-zero partial isometry }v\in M\text{ such that }\\ 
 && \nonumber rE_N(vv^*)=E_N(vv^*)r\geq sr\text{ and }\theta(qBq)v=B_o v=r v qBq.
\end{eqnarray}

By Theorem \ref{NCpreserve} and Proposition \ref{quot}, we have $\Sigma /(\Sigma \cap \Omega) = ker(\theta') /(\ker(\theta') \cap \Omega)= \ker (\theta_n)\in NC_1\cap Quot_{n-1}(\Cal C_{rss})$. Thus, by the induction hypothesis, it follows a corner of $B_o$ or $B_o'\cap rL(\Sigma)r=:C_o$ is completely atomic. However, since $B$ is diffuse then from (\ref{7.1}) and (\ref{7.2}) it follows that $B_o$ is diffuse too, and hence  a corner of $C_o$ is completely atomic.  Thus, there exists $p_o\in C_o$ nonzero projection such that  $p_oC_0p_o=\mathbb C p_o$. From (\ref{7.2}) the inclusion $B_o\vee C_o \subseteq rL( \Sigma)r$ has finite index and from Lemma \ref{ramen} it follows that the inclusion $B_op_o=B_o\vee (\mathbb C p_o)=p_o(B_o\vee C_o)p_o\subseteq p_o L( \Sigma)p_o$ has finite  index, too.

Let $\wp \in \Cal{QH}^1_{as}(\La, \pi)$ be an unbounded quasi-cocycle and let $V_t:L^2(M)\ra L^2(\tilde M)$ be corresponding Gaussian deformation as defined in the Section \ref{sec: deformations}, where $M\subseteq \tilde M$ is the Gaussian dilation of $M$. Denote by $e_M$ the orthogonal projection on $L^2 (\tilde M)$ onto $L^2(M)$. Since $C$ can always be assumed non-amenable then using the same \`{a} la Popa spectral gap argument (see for instance \cite[Theorem 3.2]{CS11}) we have that $e^{\perp}_M \cdot V_t\ra 0$ uniformly on $(B)_1$, as $t\ra 0$.  Using Proposition \ref{conv1} this further implies that  $e^{\perp}_M \cdot V_t\ra 0$ uniformly on $rv (qBq)_1$, as $t\ra 0$. Using (\ref{7.1}) and Proposition \ref{conv1} again we get that $e^{\perp}_M \cdot V_t\ra 0$ uniformly on $(B_o)_1 r vv^*$, as $t\ra 0$. Moreover, Lemma \ref{decay0} further gives that  $e^{\perp}_M \cdot V_t\ra 0$ uniformly on $(B_o)_1 r E_N(vv^*)$, as $t\ra 0$. Hence, by (\ref{7.1}) and Corollary \ref{conv'} we obtain that $e^{\perp}_M \cdot V_t\ra 0$ uniformly on $(B_o)_1 r=(B_o)_1$, as $t\ra 0$ and by Proposition \ref{conv1} again we conclude that $e^{\perp}_M \cdot V_t\ra 0$ uniformly on $(B_o)_1 p_o$, as $t\ra 0$.

Since   $B_op_o \subseteq p_oL(\Sigma)p_o$ has finite Pimsner-Popa index then using part (2) in Lemma \ref{ramen} we have that $p_o L(\Sigma ) p_o \preceq_{p_oL (\Sigma )p_o}B_op_o$. Hence one can find nonzero projections $p'_o\in  p_oL(\Sigma ) p_o$, $r_o\in B_op_o$, nonzero partial isometry $v_o\in  p_oL(\Sigma ) p_o$, and an injective unital $\star$-homomorphism $\Xi:\, p'_oL(\Sigma ) p'_o\ra r_oBr_o$ such that $\Xi(x)v_o=v_ox$, for all $x\in p'_oL(\Sigma ) p'_o$. Since $v^*_ov_o\in \mathcal Z(p'_oL(\Sigma ) p'_o)$ and $v_ov^*_o\in (\Xi(p'_o L(\Sigma)p'_o))'\cap r_oL(\Sigma ) r_o $ it follows that 
$v_ov_o^*L(\Sigma) v_ov_o^*=v_o L(\Sigma)v_o^*= \Xi(p'_o L(\Sigma)p'_o)v_ov_o^*$. Since 
$\Xi(p'_o L(\Sigma)p'_o)\subseteq p_oB_op_0$ and $e^{\perp}_M \cdot V_t\ra 0$ uniformly on $(B_o)_1 p_o$, as $t\ra 0$ then from Proposition \ref{conv1} it follows that $e^{\perp}_M \cdot V_t\ra 0$ uniformly on $(\Xi(p'_o L(\Sigma)p'_o))_1$, as $t\ra 0$. Thus using the above relations together with Proposition \ref{conv1} we further get that $e^{\perp}_M \cdot V_t\ra 0$ uniformly on $(\Xi(p'_o L(\Sigma)p'_o) v_ov^*_o)_1= (v_ov_o^*L(\Sigma) v_ov_o^*)_1$, as $t\ra 0$. Then Lemmas \ref{decay6}-\ref{decay7} and Corollary \ref{fcenter} imply the quasi-cocycle $\wp$ is bounded on $\Sigma$.  Moreover, since $\Sigma$ is normal in $\La$ then \cite[Theorem 2.1]{CSU13} further implies that $\wp$ is bounded on $\La$ and we have reached a contradiction; this settles the inductive step and hence the proof. 
$\hfill\square$
\vskip 0.08in
The authors strongly believe Theorem \ref{main1} actually holds for all groups satisfying only condition $NC_1$. A succesful strategy to show such a statement seems to depend heavily on investigating new aspects of the infinitesimal analysis of the weak deformations arising from quasicocycles (the Gaussian dilation from Section \ref{sec: deformations}). However, there are some serious technical obstacles in this direction, the most significant being the lack of uniform bimidularity as well as good averaging properties of the Gaussian dilations associated with non-proper unbounded quasicocycles.

\begin{Conjecture}Let $\G\in NC_1$ and let $L(\G)$ be its the corresponding von Neumann algebra. If $p \in L(\G)$ is  a nonzero projection, then any two diffuse, commuting subalgebras $B,C\subseteq pL(\G)p$  generate together a von Neumann subalgebra $B\vee C$ which has infinite Pimsner-Popa index in $pL(\G)p$; in particular, $L(\G)$ is prime.
\end{Conjecture}  

\subsection{Further results} Our techniques can also be used to show primeness of II$_1$  factors arising from all free ergodic probability measure preserving actions of groups in the  class $NC_1 \cap Quot (\mathcal C_{rss})$. In fact one can show the following counterpart of Theorem \ref{main1} for actions.

\begin{theorem}
Let $\G\in NC_1 \cap Quot (\mathcal C_{rss})$ and let $\G \ca X$ be a free ergodic pmp action on a probability space. Denote by $L^\infty(X)\rtimes \G$ the corresponding group measure spae von Neumann algebra. If $p \in L^\infty(X)\rtimes \G$ is  a nonzero projection, then any two diffuse, commuting subalgebras $B,C\subseteq p(L^\infty(X)\rtimes \G)p$  generate together a von Neumann subalgebra $B\vee C$ which has infinite Pimsner-Popa index in $p(L^\infty(X)\rtimes \G) p$. In particular, $L^\infty(X)\rtimes \G$ is prime. 

\end{theorem}

Finally, we remark that our techniques can be further developed to show that the groups in class $NC_1 \cap Quot (\mathcal C_{rss})$ also satisfy the unique prime decomposition phenomenon discovered by Ozawa and Popa in \cite{OP03} for bi-exact groups. Hence, our examples will add to the subsequent examples found in \cite{Pe06,CS11,SW12,Is14}. 

\begin{theorem}
For every $1\leq i\leq n$ let $\G_i \in NC_1 \cap Quot (\mathcal C_{rss})$ and for every $1\leq j\leq m$ let $P_j$ be a ${\rm II}_1$ factor. If we assume that $ L(\G_1)\bar\otimes L(\G_2)\bar\otimes \cdots \bar\otimes L(\G_n) \cong P_1\bar\otimes P_2\bar\otimes \cdots \bar\otimes P_m$ then $n\geq m$.  If we assume in addition that $P_j = L(\G'_j)$ for some $\G'_j \in NC_1 \cap Quot (\mathcal C_{rss})$ then we have $n=m$ and moreover there exist $\sigma$ a permutation of the set $\{1,\ldots ,n\}$ and positive scalars $t_i$ with $t_1t_2 \cdots t_n =1$  such that $L(\G_i)^{t_i}\cong P_{\sigma(i)}$, for all $1\leq i\leq n$.    
\end{theorem}

\subsection*{Acknowledgements}

The first and the third author are very grateful to Prof.\ Frederick Goodman for many illuminating discussions on braid groups and beyond.  The first author is grateful to Yago Antolin for his comments on right-angled Artin groups in connection with the content of this paper.

The authors are particularly grateful to the two anonymous referees for their numerous corrections and suggestions which greatly improved the exposition and the overall mathematical quality of the paper.

\end{document}